\documentclass[12pt]{amsart}
\usepackage{amssymb,amsmath,amsthm,andy}

\usepackage{fancyhdr}
\newcommand{\Toz}{T^{1,0}}
\newcommand{\Tzo}{T^{0,1}}

\newcommand{\Lb}{\bar L}
\newcommand{\Lba}{ {\bar L}^* }
\newcommand{\Lbap}{ {\bar L}^{*,+} }
\newcommand{\Lbam}{ {\bar L}^{*,-} }

\newcommand{\dbarbsp}{\dbarb^{*,+}}
\newcommand{\dbarbsm}{\dbarb^{*,-}}
\newcommand{\dbarbspm}{\dbars_{b,\pm}}

\newcommand{\Qbpm}{Q_{b,\pm}}
\newcommand{\Qbpmp}{\Qbpm(\vp,\vp)}
\newcommand{\Qbp}{Q_{b,+}}
\newcommand{\Qbpp}{\Qbp(\vp,\vp)}
\newcommand{\Qbm}{Q_{b,-}}
\newcommand{\Qbmp}{\Qbm(\vp,\vp)}
\newcommand{\Qbo}{Q_{b,0}}

\newcommand{\ob}{\,\bar\omega}
\newcommand{\omb}{\bar\omega}

\newcommand{\Cp}{\mathcal{C}^+}
\newcommand{\Cm}{\mathcal{C}^-}
\newcommand{\Co}{\mathcal{C}^0}

\newcommand{\Cpn}{\mathcal{C}^+_{\nu}}
\newcommand{\Cmn}{\mathcal{C}^-_{\nu}}
\newcommand{\Con}{\mathcal{C}^0_{\nu}}

\newcommand{\Com}{\mathcal{C}^0_{\mu}}

\newcommand{\tCp}{\tilde{\mathcal{C}}^+}

\newcommand{\tCon}{\tilde{\mathcal{C}}^0_\nu}

\newcommand{\tCpm}{ {\tilde{\mathcal{C}}}^+_{\mu} }
\newcommand{\tCmm}{\tilde{\mathcal{C}}^-_\mu}

\newcommand{\psp}{\psi^+}
\newcommand{\psm}{\psi^-}
\newcommand{\pso}{\psi^0}

\newcommand{\pspl}{\psi^+_{A}}
\newcommand{\psml}{\psi^-_{A}}
\newcommand{\psol}{\psi^0_{A}}
\newcommand{\tpspl}{{\tilde\psi}^+_{A}}

\newcommand{\pspln}{\psi^+_{\nu,A}}

\newcommand{\tpsplm}{{\tilde\psi}^+_{\mu,A}}
\newcommand{\tpsmlm}{{\tilde\psi}^-_{\mu,A}}

\newcommand{\Pso}{\Psi^0}
\newcommand{\Pspl}{\Psi^+_{A}}
\newcommand{\Psml}{\Psi^-_{A}}
\newcommand{\Psol}{\Psi^0_{A}}
\newcommand{\tPspl}{{\tilde\Psi}^+_{A}}

\newcommand{\tPsol}{{\tilde\Psi}^0_{A}}

\newcommand{\Pspla}{ (\Psi^+_{A})^* }
\newcommand{\Psmla}{ (\Psi^-_{A})^* }
\newcommand{\Psola}{ (\Psi^0_{A})^* }
\newcommand{\tPspla}{ ({\tilde\Psi}^+_{A})^* }

\newcommand{\Pspln}{\Psi^+_{\nu,A}}
\newcommand{\Psmln}{\Psi^-_{\nu,A}}
\newcommand{\Psoln}{\Psi^0_{\nu,A}}

\newcommand{\tPsoln}{{\tilde\Psi}^0_{\nu,A}}

\newcommand{\Psplan}{ (\Psi^+_{\nu,A})^* }
\newcommand{\Psmlan}{ (\Psi^-_{\nu,A})^* }
\newcommand{\Psolan}{ (\Psi^0_{\nu,A})^* }

\newcommand{\Psplm}{\Psi^+_{\mu,A}}
\newcommand{\Psmlm}{\Psi^-_{\mu,A}}
\newcommand{\Psolm}{\Psi^0_{\mu,A}}
\newcommand{\tPsplm}{{\tilde\Psi}^+_{\mu,A}}
\newcommand{\tPsmlm}{{\tilde\Psi}^-_{\mu,A}}

\newcommand{\Psplam}{ (\Psi^+_{\mu,A})^* }

\newcommand{\Jt}{ {\mathcal{J}}_{\vartheta} }
\newcommand{\tJt}{ {}^t\!{\mathcal{J}}_{\vartheta} }

\newcommand{\lp}{{\lambda^{+}}}
\newcommand{\lm}{{\lambda^{-}}}

\newcommand{\zn}{\zeta_\nu}
\newcommand{\tzn}{ {\tilde\zeta_{\nu}} }
\newcommand{\zm}{\zeta_\mu}
\newcommand{\tzm}{ {\tilde\zeta_{\mu}} }

\newcommand{\tz}{{\tilde\zeta}}

\newcommand{\normlplm}{\norm_{\lp,\lm}}
\newcommand{\normpm}{\norm_{\pm}}

\newcommand{\sumn}{\sum_{\nu}}
\newcommand{\summ}{\sum_{\mu}}

\newcommand{\vpn}{\vp^\nu}

\newcommand{\la}{\langle}
\newcommand{\ra}{\rangle}
\newcommand{\ralplm}{\rangle_{\lp,\lm}}
\newcommand{\rapm}{\rangle_{\pm}}

\newcommand{\Alp}{A_{\lp}}
\newcommand{\Alm}{A_{\lm}}
\newcommand{\Apm}{A_{\pm}}
\newcommand{\Hpm}{\mathcal{H}_{\pm}}
\newcommand{\Hpmp}{{}^\perp\Hpm}
\newcommand{\sjkp}{s_{jk}^+}
\newcommand{\sjkm}{s_{jk}^-}

\newcommand{\Boxbpm}{\Box_{b,\pm}}

\newcommand{\Gpmq}{G_{q,\pm}}

\renewcommand{\H}{\mathcal{H}}
\newcommand{\Hp}{{}^\perp\H}

\newcommand{\CRPq}{$($CR-$P_q)$}
\newcommand{\CRPnq}{$($CR-$P_{n-1-q})$}

\newtheorem{thm}{Theorem}[section]
\newtheorem{prop}[thm]{Proposition}
\newtheorem{lem}[thm]{Lemma}
\newtheorem{cor}[thm]{Corollary}

\newtheorem{defn}[thm]{Definition}

\begin{document}

\title [Compactness of the Complex Green Operator]{Compactness of the Complex Green Operator on CR-Manifolds of Hypersurface Type}
\author{Andrew Raich}

\begin{abstract}  The purpose of this article is to study compactness of the complex Green operator on CR manifolds of hypersurface
type.   We introduce  \CRPq, a potential theoretic condition on $(0,q)$-forms that generalizes Catlin's property $(P_q)$ to CR manifolds of 
arbitrary  codimension. We prove 
that if an embedded CR-manifold of hypersurface type  of real dimension at least five satisfies
\CRPq and \CRPnq, then the complex Green operator is a compact operator on the Sobolev spaces $H^s_{0,q}(M)$ 
and $H^s_{0,n-1-q}(M)$, if $1\leq q \leq n-2$ and $s\geq 0$.
We use CR-plurisubharmonic functions to build a microlocal
norm that controls the totally real direction of the tangent bundle. 
\end{abstract}

\address{
Department of Mathematical Sciences \\ 1 University of Arkansas \\ SCEN 327 \\ Fayetteville, AR 72701}


\email{araich@uark.edu}

\maketitle

%
%
\section{Introduction and Results}
In this article, we introduce property \CRPq, a potential theoretic condition on $(0,q)$-forms.
We show that if an embedded CR-manifold of hypersurface type  satisfies
\CRPq and \CRPnq, then the complex Green operator is a compact operator on the Sobolev spaces $H^s_{0,q}(M)$ 
and $H^s_{0,n-1-q}(M)$ if $1\leq q \leq n-2$.
We use CR-plurisubharmonic functions to build a microlocal
norm that controls the ``bad" direction of the tangent bundle. We first prove the closed range and compactness results on 
$L^2_{0,q}(M)$ and use an elliptic regularization argument to pass to higher Sobolev spaces.

A CR-manifold of hypersurface type $M$ is the generalization to higher codimension
of the boundary of a pseudoconvex domain. Let $\Omega\subset\C^N$ be
a pseudoconvex domain and $H$ be a holomorphic function on the closure of $\Omega$. If $h$ is the boundary value of $H$, then
$h$ satisfies the tangential Cauchy-Riemann equations $\dbarb h=0$.  As with the Cauchy-Riemann operator, $\dbarb$ 
gives rise to a complex
that is a useful tool for analyzing the behavior of forms on and near the boundary. A
CR-manifold of hypersurface type is a 
$(2n-1)$-dimensional manifold that is locally equivalent to a hypersurface in $\C^n$. The tangential
Cauchy-Riemann operator $\dbarb$ can again be thought of as the restriction of $\dbar$ to $M$.

The $L^2$-theory of $\dbarb$ has been studied when $M$ is a CR-manifold of hypersurface type. When $M$ is the boundary of a
pseudoconvex domain, it is by now classical that $\dbarb$ has closed range \cite{Kohn86,Shaw85,BoSh86}. 
More recent work by Nicoara \cite{Nic06} shows the same result holds when $M$ a CR-manifold of hypersurface type. 
The approach to analyze $\dbarb$-problems proceeds down one of two paths. One is to follow Shaw's approach
and use $\dbar$-techniques and jump formulas, and the other path is to use Kohn's ideas and develop a microlocal
analysis to control the totally real or ``bad" direction of the tangent bundle.
When $M$ is not a hypersurface,
microlocal analysis seems to be a more natural  approach, and we will use this approach.

The method that we use to
solve the $\dbarb$-equation is to introduce the Kohn Laplacian $\Boxb= \dbarbs\dbarb + \dbarb\dbarbs$ and invert it.
The inverse (modulo its null space) is called the complex Green operator and denoted $G_q$ when it acts on
$L^2_{0,q}(M)$, and the canonical solution to $\dbarb u=f$ is given by $u= \dbarbs G_q f$ (assuming $f$ satisfies the
appropriate compatibility condition, e.g., $\dbarb f=0$ when $1\leq q\leq n-2$). 
Closed range of $\dbarb$ implies that $G_q$ exists and is bounded on $L^2$,  though geometric and potential
theoretic properties of $M$ can give $G_q$ much stronger regularity properties. These additional regularity properties, however,
have only been explored when $M=\bd\Omega$ is the boundary of a pseudoconvex domain. 
In this case, subellipticity of $G_q$ holds if and only if $M$ satisfies a curvature condition called finite
type (at the symmetric level $q$ and $n-1-q$)  \cite{Cat83,Cat87,Koh02, Nic07, Dia86, Koe04, RaSt08}.
Optimal subelliptic estimates (so called maximal estimates) were obtained in \cite{Koe02}
under the additional condition that all eigenvalues of the Levi form are comparable. 
This work unifies earlier results for strictly pseudoconvex domains and for domains of finite type in $\mathbb{C}^{2}$. 
For general domains, it is known that 
if $\Omega$ admits a defining function that is plurisubharmonic at points of the boundary,
then $G_q$ preserves the Sobolev spaces $H^s(\bd\Omega)$, $s\geq 0$ \cite{BoSt91}. 
A defining function is called plurisubharmonic at the boundary when its complex Hessian at points of the boundary is positive semidefinite in all directions. For example, all convex domains admit such defining functions. 

On a pseudoconvex domain $\Omega\subset\C^N$, the $\dbar$-Neumann operator is the inverse to the
$\dbar$-Neumann Laplacian $\Box = \dbar\dbars + \dbars\dbar$ on $L^2_{0,q}(\Omega)$. 
When $q=1$, a necessary and sufficient condition for subellipticity of the $\dbar$-Neumann operator on $\Omega$
is the existence of a plurisubharmonic function
whose complex Hessian blows up proportional to a reciprocal power of the distance to the boundary \cite{Cat83,Cat87,Str97}. 
In \cite{Cat84}, Catlin introduces a weakened version of complex Hessian blowup condition and instead requires only that
there exist plurisubharmonic functions with arbitrarily large complex Hessians. He calls this condition property (P) and
its natural generalization to $(0,q)$-forms, called
$(P_q)$,  is now a well known sufficient condition for compactness of the $\dbar$-Neumann operator  (see
\cite{FuSt01,Str06} for a discussion of compactness in the $\dbar$-Neumann problem). In \cite{RaSt08}, Emil Straube and I
show that if $M=\bd\Omega$ is the boundary of a smooth, bounded, pseudoconvex domain 
and satisfies $(P_q)$ and $(P_{n-1-q})$, then
$G_q$ is a compact operator on $L^2_{0,q}(M)$. We also show that compactness of $G_q$ implies compactness of the 
$\dbar$-Neumann operator on $(0,q)$-forms on $\Omega$ and if $\bd\Omega$ is locally convexifiable then $(P_q)$
and $(P_{n-1-q})$ is equivalent to compactness of $G_q$ (see \cite{FuSt98} as well).
Our methods involve $\dbar$-techniques, a jump formula
in the spirit of Shaw (and Boas) \cite{Shaw85,BoSh86}, and a detailed study of compactness of the $\dbar$-Neumann operator on 
the annulus between two pseudoconvex domains. Applying $\dbar$-techniques to investigate the complex Green operator in the
higher codimension case investigated in this article seems to be difficult if $q>1$  because it is unknown
if $(P_q)$ is invariant under CR-equivalences (or even biholomorphisms that are not conformal mappings) if $q>1$.

%
%
%
%
The goal of this article is to generalize the compactness result of \cite{RaSt08} to the case when $M$ is a CR-manifold
of hypersurface type. We introduce property \CRPq,
a generalization of $(P_q)$ for CR-manifolds of hypersurface type, and show that it is a sufficient
condition for compactness of the complex Green operator.

Let 
\[
\H^q = \{\vp\in L^2_{0,q}(M)\cap\Dom(\dbarb)\cap\Dom(\dbarbs) : \dbarb\vp=0, \dbarbs\vp=0\} 
\]
be the space of harmonic forms and
\[
\Hp^q = \{\vp\in L^2_{0,q}(M) : ( \vp,\phi)_0 =0,\text{ for all }\phi\in\H\}.
\]

Our main result is the following theorem.
\begin{thm}\label{thm:G_q is compact}
Let $M\subset \C^N$ be a smooth, compact, orientable weakly pseudoconvex CR-manifold of hypersurface type 
of real dimension at least five 
that satisfies \CRPq and \CRPnq. If $1\leq q\leq n-2$ and $s\geq 0$,
then 
\renewcommand{\labelenumi}{(\roman{enumi})}
\begin{enumerate}
\item $\dbarb$ and $\dbarbs$ acting on $H^s_{0,q}(M)$ have closed range,
\item the complex Green operator $G_q$ exists and is a compact operator on $H^s_{0,q}(M)$,
\item $\H^q$ is finite dimensional.
\end{enumerate}
\end{thm}

The assumption that $1\leq q\leq n-2$ excludes the endpoints $q=0$ and $q=n-1$. For the endpoint case, it is not clear
what (CR-$P_0$) should be. However, one can check (in analogy to the $\dbar$-Neumann problem) that
$G_0 = \dbarbs G_1^2 \dbarb = \dbarbs G_1 (\dbarbs G_1)^*$,
and thus it follows that (CR-$P_1$) is a sufficient condition 
for compactness of $G_0$ (and $G_{n-1}$ as well).  The requirement that the dimension of $M$ is at least five is a seemingly
technical assumption concerning the eigenvalues of a Hermitian matrix. In particular, 
 and $H=(h_{jk})$ is a Hermitian, positive definite matrix,
$1\leq i,k\leq n-1$, then $(\delta_{jk}\sum_{\ell=1}^{n-1}h_{\ell\ell}-h_{jk})$ is a Hermitian, positive definite matrix if $n\geq 3$. This
fact is false when $n=2$, and this causes the three dimensional case to remain open.

The symmetric requirements at level $q$ and $n-1-q$ are necessary \cite{Koe04,RaSt08, Koh81}. 
To a $(0,q)$-form $u$ on $\bd\Omega$, 
there is an associated $(0,n-1-q)$-form $\tilde{u}$ 
(obtained through a modified Hodge-$*$  construction) such that $\|u\| \approx \|\tilde{u}\|$,  $\overline{\partial}_{b}\tilde{u}=(-1)^{q}\widetilde{(\overline{\partial}^{*}_{b}u)}$, and $\overline{\partial}_{b}^{*}\tilde{u}=(-1)^{q+1}\widetilde{(\overline{\partial}_{b}u)}$, modulo terms that are $O(\|u\|)$. Consequently, a compactness estimate holds for $(0,q)$-forms if and only if the corresponding estimate holds for $(0,n-1-q)$-forms. In view of the characterization of compactness on convex domains \cite{FuSt98}, 
such a symmetry between form levels is absent in the $\overline{\partial}$-Neumann problem. 
(The analogous construction performed for forms on $\Omega$ yields a form $\tilde{u}$ that 
in general is not in the domain of $\dbars$.)

A consequence of Theorem \ref{thm:G_q is compact} and Corollary \ref{cor:psh implies CR psh} is the following generalization
of Theorem 1.4 in \cite{RaSt08}.
\begin{cor}\label{cor:G_q is compact with P_q}
Let $M\subset \C^N$ be a smooth, compact, orientable weakly pseudoconvex CR-manifold of hypersurface type
that satisfies $(P_q)$. Then
$M$ satisfies \CRPq. In particular, if $M$ satisfies
$(P_q)$ and $(P_{n-1-q})$ and is of real dimension at least five, then the
conclusions of Theorem \ref{thm:G_q is compact} hold.
\end{cor}

I would like to thank Emil Straube for his suggestion to investigate this problem and for many 
helpful discussions. I would also like to acknowledge Siqi Fu and thank him for some insightful conversations regarding
\CRPq.

%
%
\section{Definitions and Notation}

\subsection{CR-Manifolds and the tangential Cauchy-Riemann operator $\dbarb$}

\begin{defn}\label{defn:CR mfld} 
Let $M\subset \C^N$ be a smooth manifold of real dimension $2n-1$. The \textbf{CR-structure on $M$} is given by a
complex subbundle $\Toz(M)$ of the complexified tangent bundle $T(M)\otimes \C$ that satisfies the following conditions:
\begin{enumerate}\renewcommand{\labelenumi}{(\roman{enumi})}
\item The complex dimension of each fiber of $\Toz(M)$ is $n-1$ for all $p\in M$;
\item If we define $\Tzo(M) = \overline{\Toz(M)}$, then $\Toz(M)\cap \Tzo(M) = \{0\}$;
\item If $L, L'\in \Toz(M)$ are two vector fields defined near $M$, then their commutator
$[L,L'] = LL'-L'L$ also an element of $\Toz(M)$.
\end{enumerate}
A manifold $M$ endowed with a CR-structure is called a \textbf{CR-manifold}.
\end{defn}
In the case that $M$ is a submanifold of $\C^N$, then for each $z\in\C^N$, set 
$\Toz_z(M) = \Toz_z(\C^N)\cap T_z(M)\otimes\C$ (under the natural inclusions). If the complex dimension
of $\Toz_{z}(M)$ is $n-1$ for all $z\in M$, we can then let
$\Toz(M) =\bigcup_{z\in M} T_{z}^{1,0}(M)$, and this defines the \textbf{\emph{induced CR-structure on $M$}}. Observe that conditions
(ii) and (iii) are automatically satisfied in this case.

For the remainder of this article, $M$ is a smooth, orientable 
CR-manifold of real dimension $2n-1$ embedded $\C^N$ for some $N\geq n$. 
Let $B^q(M)=\bigwedge^q(\Tzo(M))$ (the bundle of $(0,q)$ forms that consists of skew-symmetric multilinear maps of 
$\Tzo(M)^q$ into $\C$). We can
therefore choose our Riemannian metric to be the restriction on  $T(M)\otimes\C$ of the usual Hermitian inner product on
$\C^N$. 
We can define a Hermitian inner product on $B^q(M)$ by
\[
(\vp,\psi) = \int_M \langle \vp,\psi\rangle_x\, dV,
\]
where $dV$ is the volume element on $M$ and $\langle \vp,\psi \rangle_x$ is the induced inner product on $B^q(M)$.
This metric is compatible with the induced CR-structure, i.e., the vector spaces $\Toz_z(M)$ and $\Tzo_z(M)$
are orthogonal under the inner product.

The involution condition (iii) of Definition \ref{defn:CR mfld} means that there is a restriction of the de Rham exterior derivative
$d$ to $B^q(M)$, which we denote by $\dbarb$. The inner product gives rise to an $L^2$-norm $\| \cdot \|_0$, and
we also denote the closure of $\dbarb$ in this norm by $\dbarb$ (by an abuse of notation). 
In this way, $\dbarb : L^2_{0,q}(M)\to L^2_{0,q+1}(M)$ is a well-defined, closed, densely defined operator, and we define
$\dbarbs:L^2_{0,q+1}(M)\to L^2_{0,q}(M)$ to be the $L^2$-adjoint of $\dbarb$. The Kohn
Laplacian $\Boxb:L^2_{0,q}(M) \to L^2_{0,q}(M)$ is defined as
\[
\Boxb = \dbarbs\dbarb + \dbarb\dbarbs,
\]
and its inverse on $(0,q)$-forms (up to $\null(\Boxb)$) is called the \textbf{\emph{complex Green operator}} and denoted by $G_q$.

The induced CR-structure has a local basis $L_1,\dots, L_{n-1}$ for the $(1,0)$-vector fields in a neighborhood $U$ of each
point $x\in M$. Let $\omega_1,\dots,\omega_{n-1}$ be the dual basis of $(1,0)$-forms that satisfy $\langle \omega_j, L_k\rangle 
=\delta_{jk}$. Then $\Lb_1,\dots, \Lb_{n-1}$ is a local basis for the $(0,1)$-vector fields with dual basis $\ob_1,\dots,\ob_{n-1}$
in $U$. Also, $T(U)$ is spanned by $L_1,\dots,L_{n-1}$, $\Lb_1,\dots,\Lb_{n-1}$ and one more vector $T$ taken to be 
purely imaginary (so $\bar T = -T$). Let $\gamma$ be the purely imaginary global $1$-form on $M$ that
annihilates $\Toz(M)\oplus \Tzo(M)$ and is normalized so that $\langle \gamma, T \rangle =-1$.

\begin{defn}\label{defn:Levi form}
The \textbf{Levi form at a point $x\in M$} is the Hermitian form given by $\langle d\gamma_x, L\wedge \bar L'\rangle$
where $L, L'\in \Toz_x(U)$, $U$ a neighborhood of $x\in M$. We call $M$ \textbf{weakly pseudoconvex} if there exists a
form $\gamma$ such that the Levi form is positive semi-definite at all $x\in M$ and \textbf{strictly pseudoconvex} if there
is a form $\gamma$ such that the Levi form is positive definite at all $x\in M$.
\end{defn}

\subsection{Property \CRPq and CR-plurisubharmonic functions}


\begin{defn} A smooth function $\vp:\Omega\to\C$ is called \textbf{plurisubharmonic on $(0,q)$-forms} if the 
sum of any $q$ eigenvalues of the complex Hessian of $\vp$ at $z\in\Omega$ is at least $C\geq 0$. The constant $C$ is the {constant of
plurisubharmonicity}(of $\vp$ at $z$). 
\end{defn}

\begin{defn} A surface $S \subset \R^k$ satisfies \textbf{property $($P${}_{q})$} if for every $C>0$, there exists a function $\vp$ and a neighborhood
$U\supset S$ so that $0\leq \vp \leq 1$ and $\vp$
is plurisubharmonic on $(0,q)$-forms on $U$ with plurisubharmonicity constant $C$.
\end{defn}

As discussed above, property $(P_q)$ has played a crucial role in the development of the compactness theory for the $\dbar$-Neumann operator
and now we define its analog for the compactness theory of the complex Green operator on CR-manifolds of hypersurface type.

\begin{defn}\label{defn:q CR psh}
Let $M$ be a CR-manifold. A real-valued $\mathcal{C}^\infty_c$ function $\lam$ defined in a neighborhood of $M$ is called
\textbf{strictly CR-plurisubharmonic on $(0,q)$-forms} if there exist constants $A_0, A_\lam>0$ so that for any orthonormal 
$Z_j\in \Toz(M)$, $1\leq j\leq q$,
\[
\sum_{j=1}^q \Big\la \frac 12\Big( \p_b\dbarb\lam - \dbarb\p_b\lam\Big) + A_0d\gamma, Z_j\wedge \ZZ_j\Big\ra
\geq A_\lam
\]
where $d\gamma$ is the invariant expression of the Levi form. $\lam$ is called \textbf{weakly CR-plurisubharmonic 
on $(0,q)$-forms} if $A_\lam\geq 0$. $A_\lam$ is called the \textbf{CR-plurisubharmonicity constant}.
\end{defn}
CR-plurisubharmonic functions were first introduced by Nicoara \cite{Nic06} to prove closed range of $\dbarb$ on CR manifolds
of hypersurface type.

\begin{defn} A surface $S \subset \R^k$ satisfies \textbf{property $($CR-P${}_{q})$}
if for every $A>0$, there exists a function $\lam$ and a neighborhood
$U\supset S$ so that $0\leq \lam \leq 1$ and $\lam$
is CR-plurisubharmonic on $(0,q)$-forms on $U$ with CR-plurisubharmonicity constant $A$.
\end{defn}
Appendix \ref{app:multilinear algebra} contains results multilinear algebra that help to explain the relationship of the definitions
of $(P_q)$ and \CRPq.

In this article, constants with no subscripts may depend on $n$, $N$, $M$ but not the CR-plurisubharmonic functions 
$\lp$, $\lm$, or any quantities associated with
$\lp$ or $\lm$. Those constants will be denoted with an,  $\lp$, $\lm$, or $\pm$ in the subscript. The constant $A$ will be reserved
the constant in the construction of pseudodifferential operators in Section \ref{sec:local coordinates} (though $A$ with subscripts will not).

%
%
\section{Computations in Local Coordinates}\label{sec:local coordinates}

\subsection{Local coordinates and CR-plurisubharmonicity} 

The microlocal analysis that we will use relies the existence of suitable local coordinates. The first such result is
Lemma 3.2 from \cite{Nic06}, recorded here as the following result.
\begin{lem}\label{lem:good coordinates}
Let $M$ be a compact smooth, $(2n-1)$-dimensional weakly pseudoconvex CR-manifold of hypersurface type embedded in a complex
space $\C^N$ such that $N\geq n$ and endowed with an induced CR-structure. For each point
$P\in M$, there exists
a  neighborhood $U$ so that $M\cap U$ is CR-equivalent to a hypersurface in $\C^n$. Additionally, on $U$ there is 
a local orthonormal basis $L_1,\dots,L_n$, $\Lb_1,\dots,\Lb_n$ of the $n$-dimensional
complex bundle containing $TM$:\\
(i) $L_j\big|_P = \frac{\p}{\p w_j}$ for $1\leq j\leq n$ where $(w_1,\dots,w_N)$ are the coordinates of $\C^N$, and \\
(ii) $[L_j,\Lb_k]\big|_P = c_{jk}T$ where $T = L_n-\Lb_n$ and $c_{jk}$ are the coefficients of the Levi form
in $L_1,\dots,L_{n-1},\Lb_1,\dots,\Lb_{n-1},T$, a local basis for $TM$.
\end{lem}

The local coordinates from Lemma \ref{lem:good coordinates} allow us to make a careful comparison of the Levi form
with its $\dbarb$-analog. 
\begin{prop}\label{prop:dbar vs. dbar_b on M}
Let $M$ be as in Lemma \ref{lem:good coordinates}. If $\lambda$ is a smooth function near $M$, 
$L\in \Toz(M)$, and $\nu = L_n + \Lb_n$ is the ``real normal" to $M$,
then on $M$,
\[
\left\la\frac12\Big(\p\dbar\lam-\dbar\p\lam\Big), L\wedge\Lb\right\ra
- \left\la\frac12\Big(\p_b\dbarb\lam-\dbarb\p_b\lam\Big), L\wedge\Lb\right\ra
= \frac 12\nu\{\lam\} \la d\gamma, L\wedge \bar L\ra
\]
\end{prop}

\begin{proof}
Using Lemma \ref{lem:good coordinates}, there exists a basis of $\C T(\C^N)$ given by
$L_1,\dots,L_N,\Lb_1,\dots,\Lb_N$ so that $L_1,\dots,L_{n-1}$ and $\Lb_1,\dots,\Lb_{n-1}$ are a basis of
$\Toz(M)$ and $\Tzo(M)$, respectively, $T = L_n-\Lb_n\in TM$ is a purely imaginary tangent vector, and
$\nu = L_n+\Lb_n$ is the ``real normal" tangent vector to $M$. Let $\omega_1,\dots,\omega_N, \omb_1,\dots,\omb_N$
be the dual cotangent vectors to $L_1,\dots,L_N,\Lb_1,\dots,\Lb_N$, respectively. Assume that the coordinates are centered
around $P\in M$ in sense of Lemma \ref{lem:good coordinates}.

Recall that $\p\dbar = -\dbar\p$, so $\p\dbar = \frac 12(\p\dbar - \dbar\p)$. We now compute
\begin{align}
\p\dbar\lam = \p(\sum_{k=1}^N \Lb_k\lam\ob_k) \label{eqn:d dbar}
&= \sum_{j,k=1}^N L_j\Lb_k\lam\,\omega_j\wedge\omb_k +  \sum_{\ell=1}^N\Lb_\ell\lam \,\p\omb_\ell 
\end{align}
Also,
\begin{align} \label{eqn:dbar d}
\dbar\p\lam = \dbar(\sum_{j=1}^N L_j\lam\,\omega_j) 
&= -\sum_{j,k=1}^N \Lb_k L_j\lam\,\omega_j\wedge\omb_k +  \sum_{\ell=1}^N L_\ell\lam \,\dbar\omega_\ell
\end{align}

Let $L = \sum_{j=1}^{n-1} \xi_j L_j$ be a complex tangent vector on $M$. Then
\begin{align*}
\langle \dbar\omega_\ell, L_j \wedge \bar L_k \rangle \big|_P
&= L_j\{ \la \omega_\ell , \bar L_k \ra \}\big|_P - \bar L_j \{ \langle \omega_\ell, L_k\ra \} \big|_P
- \langle \omega_\ell, [L_j,\Lb_k] \rangle\big|_P \\
&= - \la(\omega_\ell, c_{jk} T\ra\big|_P= -\delta_{\ell n} c_{jk}(P).
\end{align*}
Similarly, since $T= L_n-\Lb_n$,
\begin{align*}
\langle \p\omb_\ell, L_j \wedge \bar L_k \rangle \big|_P
&= L_j\{ \la \omb_\ell , \bar L_k \ra \}\big|_P - \bar L_j \{ \langle \omb_\ell, L_k\ra \} \big|_P
- \langle \omb_\ell, [L_j,\Lb_k] \rangle\big|_P \\
&= - \la\omb_\ell, c_{jk} T\ra\big|_P= \delta_{\ell n} c_{jk}(P).
\end{align*}
Consequently, for $1\leq j,k\leq n-1$,
\[
\la \sum_{\ell=1}^N\big( \Lb_\ell\lam \,\p\omb_\ell - L_\ell\lam \,\p\omb_\ell \big), L_j \wedge \Lb_k\ra\big|_P
= (\Lb_n\{\lam\}+ L_n\{\lam\}) c_{jk}(P) = c_{jk}(P) \nu(\lam)\big|_P
\]
If $K = \sum_{j=1}^N \xi_j L_j + \sum_{k=1}^N \zeta_k \Lb_k$, then
\[
\la \omega_j\wedge\omb_k,K\wedge \bar K\ra = \omega_j(K)\omb_k(\bar K) - \omega_j (\bar K)\omb_k(K) 
= \xi_j\bar\xi_k - \bar\zeta_j\zeta_k.
\]
Putting the equations together, for $L = \sum_{j=1}^{n-1}\xi_j L_j$, we have that
\[
\la\p\dbar\lam, L\wedge\Lb \ra\Big|_P = \left\la\frac12\Big(\p\dbar\lam-\dbar\p\lam\Big), L\wedge\Lb\right\ra\bigg|_P
=\frac 12\sum_{j,k=1}^{n-1}\Big[ L_j\Lb_k\lam + \Lb_k L_j\lam + \nu(\lam)\, c_{jk} \Big]\xi_j \bar\xi_k\bigg|_P. 
\]
To understand $L_j\Lb_k\lam + \Lb_k L_j\lam$, we expand the vector fields in the ambient $\C^N$ coordinates. In coordinates,
\[
L_j = \sum_{\ell=1}^N a_\ell^j \frac{\p}{\p w_\ell}.
\]
This means
\[
L_j\Lb_k  - \sum_{\ell,\ell'=1}^N a_\ell^j \bar a^k_{\ell'} \frac{\p^2}{\p w_\ell \p\w_{\ell'}}
=\sum_{\ell,\ell'=1}^N a^j_\ell \frac{\p \bar a^k_{\ell'}}{\p w_\ell} \frac{\p}{\p\w_{\ell'}} \in \Tzo(\C^N)
\]
and
\[
\Lb_k L_j  - \sum_{\ell,\ell'=1}^N a_\ell^j \bar a^k_{\ell'} \frac{\p^2}{\p w_\ell \p\w_{\ell'}}
=\sum_{\ell,\ell'=1}^N  \bar a^k_{\ell'} \frac{\p a^j_\ell}{\p \w_{\ell'}} \frac{\p}{\p w_{\ell}}\in \Toz(\C^N)
\]
Since $[L_j,\Lb_k]\big|_P = L_j\Lb_k - \Lb_k L_j\big|_P = c_{jk} T\big|_P = c_{jk}(L_n-\Lb_n)\big|_P$, it follows that
\[
\sum_{\ell,\ell'=1}^N  \bar a^k_{\ell'} \frac{\p a^j_\ell}{\p \w_{\ell'}} \frac{\p}{\p w_{\ell}} \bigg|_P
= -c_{jk}L_n \Big|_P
\]
and
\[
\sum_{\ell,\ell'=1}^N a^j_\ell \frac{\p \bar a^k_{\ell'}}{\p w_\ell} \frac{\p}{\p\w_{\ell'}} \bigg|_P = 
- c_{jk}\Lb_n\Big|_P.
\]
Thus, since $L_j \big|_P = \frac{\p}{\p w_j}$ by Lemma \ref{lem:good coordinates},
\[
\big(L_j\Lb_k\lam + \Lb_k L_j\lam\big)\Big|_P = \frac{\p^2\lam(P)}{\p w_j\p\w_k} - \nu(\lam)\Big|_P c_{jk}(P).
\]
Finally,
\begin{align*}
\left\la\frac12\Big(\p\dbar\lam-\dbar\p\lam\Big), L\wedge\Lb\right\ra\Big|_P
&= \frac 12 \sum_{j,k=1}^{n-1}\left[ \Big( \frac{\p^2\lam}{\p w_j\p\w_k} + c_{jk}\big( \nu(\lam) - \nu(\lam)\big) \Big) \xi_j\bar\xi_k
\right] \Bigg|_P \\
&=\frac 12 \sum_{j,k=1}^{n-1}  \frac{\p^2\lam}{\p w_j\p\w_k} \xi_j\bar\xi_k \Big|_P
\end{align*}
The calculation of $\la\frac12\big(\p_b\dbarb\lam-\dbarb\p_b\lam\big), L\wedge\Lb\ra\big|_P$ is performed identically except that
the sums in \eqref{eqn:d dbar} and \eqref{eqn:dbar d} only go to $n-1$ and not to $N$. The result is that
\[
\left\la\frac12\Big(\p_b\dbarb\lam-\dbarb\p_b\lam\Big), L\wedge\Lb\right\ra\Big|_P
= \frac 12 \sum_{j,k=1}^{n-1}\left[ \Big( \frac{\p^2\lam}{\p w_j\p\w_k} - c_{jk} \nu(\lam) \Big) \xi_j\bar\xi_k
\right] \Bigg|_P.
\]
Consequently,
\begin{align*}
\left\la\frac12\Big(\p\dbar\lam-\dbar\p\lam\Big), L\wedge\Lb\right\ra\Bigg|_P 
- \left\la\frac12\Big(\p_b\dbarb\lam-\dbarb\p_b\lam\Big), L\wedge\Lb\right\ra\Bigg|_P
&= \frac 12 \sum_{j,k=1}^{n-1} c_{jk} \nu(\lam) \, \xi_j\bar\xi_k \bigg|_P\\
&= \frac 12 \nu(\lam) \la d\gamma, L\wedge \bar L\ra\Big|_P
\end{align*}
However, $T$ and $d\gamma$ are globally defined quantities and $P$ was arbitrary, so on $M$,
\[
\left\la\frac12\Big(\p\dbar\lam-\dbar\p\lam\Big), L\wedge\Lb\right\ra
- \left\la\frac12\Big(\p_b\dbarb\lam-\dbarb\p_b\lam\Big), L\wedge\Lb\right\ra
= \frac 12\nu\{\lam\} \la d\gamma, L\wedge \bar L\ra
\]
\end{proof}
We can already see from Proposition \ref{prop:dbar vs. dbar_b on M} the importance of CR-plurisubharmonic functions.
On a compact (smooth) manifold, $\nu\{\lam\}$ will be a bounded quantity, and multiples of 
Levi-form are controlled by CR-plurisubharmonicity.

If $\lam$ is smooth function defined near $P\in M$, let $\lam_{jk}$ satisfy
\[
\p\dbar \lam = \sum_{j,k=1}^N \lam_{jk} \,\omega_j\wedge\omb_k.
\]
Also, let $\I_q = \{J= (j_1,\dots,j_q)\in\N^q : 1 \leq j_1 < \cdots < j_q \leq n\}$ and $\I_q' = \{J\in \I_q : j_q<n\}$. 

As a  of Proposition \ref{prop:dbar vs. dbar_b on M} and Lemma \ref{lem:multilinear algebra}, 
we learn that functions that are plurisubharmonic 
on $(0,q)$-forms near $M$ are CR-plurisubharmonic on $(0,q)$-forms.
\begin{cor}\label{cor:psh implies CR psh}
Let $M$ be as in Lemma \ref{lem:good coordinates}. If $\lam$ is a smooth, real-valued function that is 
plurisubharmonic on $(0,q)$-forms near $M$ and has CR-plurisubharmonicity constant $A_\lam$, 
then $\lam$ is CR-plurisubharmonic on $(0,q)$-forms with CR-plurisubharmonicity constant $A_\lam$.
\end{cor}
\subsection{Pseudodifferential Operators}
From Lemma \ref{lem:good coordinates}, there exists a finite cover $\{ U_\nu\}_{\nu}$ so each $U_{\nu}$ has a special boundary system and
can be parameterized by a hypersurface in $\C^n$ ($U_\nu$ may be shrunk as necessary). To set up the microlocal analysis,
we need to define the appropriate pseudodifferential operators on each $U_{\nu}$. Let
$\xi = (\xi_1,\dots,\xi_{2n-2},\xi_{2n-1}) = (\xi',\xi_{2n-1})$ be the coordinates in Fourier space so that
$\xi'$ is dual to the part of $T(M)$ in the maximal complex subspace (i.e., $\Toz(M)\oplus \Tzo(M)$) and 
$\xi_{2n-1}$ is dual to the totally real part of $T(M)$, i.e., the ``bad" direction $T$. Define
\begin{align*}
\Cp &= \{ \xi : \xi_{2n-1} \geq \frac 12 |\xi'| \text{ and } |\xi|\geq1\};\\
\Cm &= \{\xi : -\xi\in \Cp\};\\
\Co &= \{\xi : -\frac 34|\xi'| \leq \xi_{2n-1}\leq \frac 34 |\xi'|\} \cup \{\xi : |\xi|\leq 1\}.
\end{align*}
Note that $\Cp$ and $\Cm$ are disjoint, but both intersect $\Co$ nontrivially. Next, we define functions on 
$\{|\xi| : |\xi|^2 =1\}$. Let
\begin{align*}
\psp(\xi) &= 1 \text{ when } \xi_{2n-1}\geq \frac 34|\xi'| \text{ and } \supp \psp \subset \{\xi : \xi_{2n-1}\geq \frac 12|\xi'|\}; \\
\psm(\xi) &= \psp(-\xi); \\
\pso(\xi) &\text{ satisfies } \pso(\xi)^2 = 1- \psp(\xi)^2 - \psm(\xi)^2.
\end{align*}
Extend $\psp$, $\psm$, and $\pso$ homogeneously outside of the unit ball, i.e., if $|\xi|\geq 1$, then
\[
\psp(\xi) = \psp(\xi/|\xi|),\ \psm(\xi) = \psm(\xi/|\xi|),\text{ and } \pso(\xi) = \pso(\xi/|\xi|). 
\]
Also, extend $\psp$, $\psm$, and $\pso$ smoothly inside the unit ball so that $(\psp)^2+(\psm)^2 + (\pso)^2 =1$. Finally,
for $A$ to be chosen later, define
\[
\pspl(\xi) = \psi(\xi/A),\ \psml(\xi) = \psm(\xi/A),\text{ and }\psol(\xi) = \pso(\xi/A).
\]
Next, let $\Pspl$, $\Psml$, and $\Pso$ be the pseudodifferential operators of order zero with symbols
$\pspl$, $\psml$, and $\psol$, respectively. The equality $(\pspl)^2+(\psml)^2 + (\psol)^2 =1$ implies that
\[
\Pspla\Pspl + \Psola\Psol + \Psmla\Psml = Id.
\]
We will also have use for pseudodifferential operators  that ``dominate" a given pseudodifferential operator. Let
$\psi$ be cut-off function and $\tilde\psi$ be another cut-off function so that $\tilde\psi|_{\supp \psi} \equiv 1$. If $\Psi$ and $\tilde\Psi$ are pseudodifferential operators with symbols $\psi$ and $\tilde\psi$, respectively, then we say that
$\tilde\Psi$ dominates $\Psi$.

For each $U_\nu$, we have a local CR-equivalence to a hypersurface  in $\C^n$, and we can define
$\Pspl$, $\Psml$, and $\Psol$ to act on functions or forms supported in $U_\nu$, so let $\Pspln$, $\Psmln$, and $\Psoln$
be the pseudodifferential operators of order zero defined on $U_\nu$ and $\Cpn$, and $\Cmn$, and $\Con$ be the regions of 
$\xi$-space dual to $U_\nu$ on which the symbol of each of those pseudodifferential operators is supported. Then it follows that:
\[
\Psplan\Pspln + \Psolan\Psoln + \Psmlan\Psmln = Id.
\]
Let $\tPsplm$ and $\tPsmlm$ be pseudodifferential operators that dominate $\Psplm$ and $\Psmlm$, respectively
(where $\Psplm$ and $\Psmlm$ are defined on some $U_\mu$). If $\tCpm$ and $\tCmm$ are the supports of
$\tPsplm$ and $\tPsmlm$, respectively, then we can choose $\{U_\mu\}$, $\tpsplm$, and $\tpsmlm$ so that the following result holds.

\begin{lem}\label{lem: neighborhood intersection lemma}
Let $M$ be a compact, orientable, embedded CR-manifold. There is a finite open covering $\{U_\mu\}_\mu$ of $M$ so that
if $U_\mu, U_\nu\in \{U_\mu\}$ have nonempty intersection, then
there exists a diffeomorphism $\vartheta$ between $U_\nu$ and $U_\mu$ with Jacobian
$\Jt$ so that:
\begin{enumerate}
\item $\tJt(\tCpm)\cap \Cmn = \emptyset$ and $\Cpn \cap \tJt(\tCmm) = \emptyset$ where $\tJt$ is the inverse of the transpose
of $\Jt$;

\item Let ${}^\vartheta\!\Psplm$, ${}^\vartheta\!\Psmlm$, and ${}^\vartheta\!\Psolm$ be the transfers of
$\Psplm$, $\Psmlm$, and $\Psolm$, respectively via $\vartheta$. Then on 
$\{\xi : \xi_{2n-1} \geq \frac 45 |\xi'| \text{ and } |\xi|\geq (1+\ep)A\}$, then principal symbol
of ${}^\vartheta\!\Psplm$ is identically 1, on 
$\{\xi : \xi_{2n-1} \leq -\frac 45 |\xi'| \text{ and } |\xi|\geq (1+\ep)A\}$, then principal symbol
of ${}^\vartheta\!\Psmlm$ is identically 1, and on
$\{\xi : -\frac 13 \xi_{2n-1} \geq \frac 13 |\xi'| \text{ and } |\xi|\geq (1+\ep)A\}$, then principal symbol
of ${}^\vartheta\!\Psolm$ is identically 1, where $\ep>0$ and can be very small;

\item Let ${}^\vartheta\!\tPsplm$, ${}^\vartheta\!\tPsmlm$ be the transfers via $\vartheta$ of
$\tPsplm$ and $\tPsmlm$, respectively. Then the principal symbol of ${}^\vartheta\!\tPsplm$ is identically 1 on
$\Cpn$ and the principal symbol of ${}^\vartheta\!\tPsmlm$ is identically 1 on $\Cmn$;

\item $\tCpm\cap\tCmm = \emptyset$.
\end{enumerate}
\end{lem}
We will suppress the left superscript $\vartheta$ as it should be clear from the context which pseudodifferential operator
must be transferred. The proof of this lemma is contained in Lemma 4.3 and its subsequent discussion in \cite{Nic06} .

\subsection{Norms}
We have a volume form $dV$ on $M$, and we define the following inner products and norms on functions (with their natural
generalizations to forms). Let $\lp$ and $\lm$ be functions defined on $M$.
\begin{align*}
(\phi,\vp)_0 &= \int_M \phi \bar\vp\, dV, \text{ and } \|\vp \|_0^2 = (\vp,\vp)_0 \\
(\phi,\vp)_\lp &= \int_M \phi \bar\vp\, e^{-\lp}\, dV, \text{ and } \|\vp \|_\lp^2 = (\vp,\vp)_\lp \\
(\phi,\vp)_\lm &= \int_M \phi \bar\vp\, e^{\lm} \, dV, \text{ and } \|\vp \|_\lm^2 = (\vp,\vp)_\lm. 
\end{align*}
If $\vp = \sum_{j\in\I_q'} \vp_j \ob_J$, then we use the common shorthand $\|\vp \| = \sum_{j\in\I_q'} \|\vp_J \|$ where $\| \cdot \|$ represents
a generic norm norm applied to $\vp$. 

We also need a norm that is well-suited for the microlocal arguments. Let $\{\zn\}$ be a partition of unity
subordinate to the covering $\{U_\nu\}$ satisfying $\sum_{\nu} \zn^2=1$. Also, for each $\nu$, let $\tzn$ be a cutoff function
that dominates $\zn$ so that $\supp\tzn \subset U_\nu$. Then we define the global inner product and norm as follows:
\begin{multline*}
\langle \phi,\vp \ralplm = \la \phi, \vp \rapm 
= \sumn \big( ( \tzn\Pspln\zn \phi^\nu, \tzn\Pspln\zn \vp^\nu )_{\lp} \\
+ (\tzn \Psoln \zn \phi^\nu ,\tzn \Psoln \zn \vp^\nu )_0
+ (\tzn \Psmln \zn \phi^\nu, \tzn \Psmln \zn \vpn )_{\lm} \big)
\end{multline*}
and
\[
\norm \vp \normlplm^2 = \norm \vp \normpm^2 
= \sumn \big( \| \tzn\Pspln\zn \vp^\nu \|_{\lp}^2 + \|\tzn \Psoln \zn \vp^\nu \|_0^2
+ \|\tzn \Psmln \zn \vpn \|_{\lm}^2 \big), 
\]
where $\vpn$ is the form $\vp$ expressed in the local coordinates on $U_\nu$. The superscript $\nu$ will often be omitted.

For a form $\vp$ supported on $M$, the Sobolev norm of order $s$ is given by the following:
\[
\|\vp\|_s^2 = \sumn \|\tzn\Lambda^s\zn \vp^\nu \|_{0}^2
\]
where $\Lambda$ is defined to be the pseudodifferential operator with symbol $(1+|\xi|^2)^{1/2}$.

It will be essential for us to pass from a the unweighted $L^2$-norm on $M$ and the microlocal norm defined above. The following
lemma says that we can do this without any loss of information.
\begin{lem}\label{lem:equivalence of norms}
Let $\lp$, $\lm$ be smooth functions on $M$ with $0\leq \lp,\lm\leq 1$. Then there exist constants $C_1, C_2>0$ 
so that
\[
C_1 \|\vp\|_0^2 \leq \norm \vp \norm_{\pm}^2 \leq C_2 \|\vp\|_{0}^2 
\]
\end{lem}

\begin{proof} It is enough to check this when $\vp$ is a function.
Since $0 \leq \lp,\lm\leq 1$,
\[
\norm \vp\normpm^2 \leq e\sumn\big( \|\tzn\Pspln \zn\vpn\|_{0}^2 + \|\tzn\Psoln \zn\vpn\|_{0}^2 + \|\tzn\Psmln \zn\vpn\|_{0}^2\big). 
\]
We can express $\tzn\Pspln \zn\vpn = \Pspln \zn\vpn - (1-\tzn)\Pspln \zn\vpn$. $(1-\tzn)\Pspln \zn$ is infinitely smoothing, but 
using this bound would lead to a constant depending on $A$. We wish to avoid constants depending on $A$. Observe that
\begin{multline*}
(1-\tzn(x)) \Pspln\zn(x)\vpn(x) 
= \frac{1}{(2\pi)^{2n-1}} (1-\tzn(x)) \int_{\R^{2n-1}} e^{ix\cdot\xi} \pspln(\xi) \widehat{\zn\vpn}(\xi)\, d\xi \\
= \frac{1}{(2\pi)^{2n-1}}\int_{\R^{2n-1}} \vpn(y) \int_{\R^{2n-1}}(1-\tzn(x))\zn(y)e^{i(x-y)\cdot\xi} 
\pspln(\xi) \,d\xi\, dy
\end{multline*}
Define $K(x,y)= \frac{1}{(2\pi)^{2n-1}} \int_{\R^{2n-1}}(1-\tzn(x))\zn(y)e^{i(x-y)\cdot\xi} 
\pspln(\xi) \,d\xi$.
By integration by parts, for any multiindex $\alpha$,
\[
K(x,y) = (1-\tzn(x))\zn(y) \frac{(-i)^\alpha}{(2\pi(x-y)^\alpha)^{2n-1}}\int_{\R^{2n-1}}e^{i(x-y)\cdot\xi} 
D^\alpha \pspln(\xi)\, d\xi.
\]
Recall that $\pspln(\xi) = \psp(\xi/A)$, so requiring that $A\geq 1$ means that
$|D^\alpha \pspln(\xi)| \leq C_\alpha$ where $C_\alpha$ does not depend on $A$. However,
$\supp(1-\tzn) \cap \supp\zn = \emptyset$, so for any $N$, there exists $C_N$ so that
\[
|K(x,y)| \leq |1-\tzn(x)||\zn(y)| \frac {C_N}{(1+|x-y|)^N},
\]
where $C_N$ does not depend on $A$. Consequently,
\[
\|(1-\tzn) \Pspln\zn\vpn(x) \|_0^2 \leq \tilde C \|\zn\vpn\|_0^2.
\]
The range of $\Pspln \zn$ is not $L^2(U_\nu)$ but $L^2(\R^{2n-1})$, but this problem is mitigated by the fact that
$\Pspln\zn$ is a smoothing operator outside of
$\Dom(\zn)$. Also, $\Pspl\zn$ is a contraction on $L^2(\R^{2n-1})$, so 
\[
\|\tzn\Pspln\zn\vpn\|_0^2 \leq 2 \|\Pspln\zn\vpn\|_0^2 + 2\|(1-\tzn)\Pspln\zn\vpn\|_0^2
\leq C_+\|\zn\vpn\|_0^2 
\]
for some $C$ independent of $A$. By (possibly) increasing $C$, a similar bound
will also hold for  for $\Psoln$ and $\Psmln$. The upper bound of the 
lemma therefore follows (since the sum over $\nu$ is finite and $0\leq\zn\leq1$). 

We now show the lower bound. Note that $\sumn \zn^2 = 1 = \sumn\tzn\zn^2$. Consequently,
\begin{align*}
&\|\vp\|_0^2= \big(\sumn\zn^2\vp,\vp\big)_0 = \sumn \|\zn\vpn\|_0^2  \\
&=\sumn \Big( \big( \Psplan\Pspln + \Psolan\Psoln + \Psmlan\Psmln\big)\vpn,\vpn\Big)_0 \\
&=\sumn \Big( \|(\tzn + (1-\tzn))\Pspln\zn\vpn\|_0^2 + \|(\tzn + (1-\tzn))\Psoln\zn\vpn\|_0^2 
+ \|(\tzn + (1-\tzn))\Psmln\zn\vpn\|_0^2  \Big)
\end{align*}
However, $\|(\tzn + (1-\tzn))\Pspln\zn\vpn\|_0^2 \leq 2(\|\tzn\Pspln\zn\vpn\|_0^2 + \|(1-\tzn)\Pspln\zn\vpn\|_0^2)$, and
$\Pspln\zn\vpn$ is pseudolocal (indeed,  $(1-\tzn))\Pspln\zn\vpn$ is infinitely smoothing), so $\|\tzn\Pspln\zn\vpn\|_0^2$ controls 
$\| (1-\tzn)\Pspln\zn\vpn\|_0^2$ and similarly for $\Psmln$ and $\Psoln$. As a result,
\begin{align*}
\|\vp\|_0^2 &\leq C\sumn\Big( \|\tzn\Pspln\zn\vpn\|_0^2 + \|\tzn\Psoln\zn\vpn\|_0^2 
+ \|\tzn\Psmln\zn\vpn\|_0^2 \Big) \\
&\leq C \sumn\Big( \|\tzn\Pspln\zn\vpn\|_\lp^2 + \|\tzn\Psoln\zn\vpn\|_0^2 
+ \|\tzn\Psmln\zn\vpn\|_\lm^2 \Big)
\end{align*}
since $\lp$ and $\lm$ are positive, bounded, and bounded away from zero.
\end{proof}

The meaning of Lemma \ref{lem:equivalence of norms} is that $\norm \vp\normpm \sim \|\vp\|_0^2$ with constants independent 
of $A$, 
so the Riesz Representation Theorem implies the following corollary (see Corollary 4.6 in
\cite{Nic06}).
\begin{cor} There exists a self-adjoint operator $E_{\lp,\lm} = E_{\pm}$ so that
\[
(\vp,\phi)_0 = \la \vp, E_\pm \phi\rapm
\]
for any two forms $\vp$ and $\phi$ in $L^2(M)$. $E_\pm$ is the inverse of 
\[
F_{\pm} = \sumn\Big( \zn \Psplan\tzn e^{-\lp} \tzn\Pspln\zn + \zn \Psolan \tzn^2 \Psoln\zn
+ \zn \Psmlan\tzn e^{\lm} \tzn\Psmln\tzn\Big).
\]
\end{cor}
$E_\pm$ and $F_\pm$ are bounded in $L^2(M)$ independently of $A\geq1$ since $0\leq \lp,\lm \leq 1$.

\subsection{$\dbarb$ and its adjoints}
If $f$ is a function on $M$, in local coordinates,
\[
\dbarb f = \sum_{j=1}^{n-1} \Lb_j f \ob_j,
\]
while if $\vp$ is a $(0,q)$-form, there exist functions $m_K^J$ so that
\[
\dbarb\vp = \sum_{\atopp{J\in\I_q'}{K\in\I_{q+1}'}} \sum_{j=1}^{n-1} \ep^{jJ}_K \Lb_j \vp_J\ob_K
+ \sum_{\atopp{J\in\I_q'}{K\in\I_{q+1}'}} \vp_J m^J_K \ob_K.
\]
Let $\Lba_j$ be the adjoint of $\Lb_j$ in $(\cdot, \cdot)_0$, $\Lbap_j$ be the adjoint of $\Lb_j$ in
$(\cdot,\cdot)_\lp$, and $\Lbam_j$ be the adjoint of $\Lb_j$ in $(\cdot,\cdot)_\lm$. Then we define
$\dbarbs$, $\dbarbsp$, and $\dbarbsm$ to be the adjoints of $\dbarb$ in $L^2(M)$, $L^2(M,e^{-\lp})$,
and $L^2(M,e^{\lm})$, respectively.
On a $(0,q)$-form $\vp$, we have (for some functions $f_j\in C^\infty(U)$)
\begin{align*}
\dbarbs \vp &= \sum_{\atopp{I\in\I_{q-1}'}{J\in\I_q'}}\sum_{j=1}^{n-1} \ep^{jI}_J \Lba_j \vp_J \ob_I
+ \sum_{\atopp{I\in\I_{q-1}'}{J\in\I_q'}} \overline{m^I_J} \vp_J  \ob_I \\
&= -\sum_{\atopp{I\in\I_{q-1}'}{J\in\I_q'}}\sum_{j=1}^{n-1}\ep^{jI}_J\big( L_j \vp_J + f_j \vp_J\big)\omb_I
+ \sum_{\atopp{I\in\I_{q-1}'}{J\in\I_q'}} \overline{m^I_J} \vp_J  \ob_I
\end{align*}
\begin{align} \label{eqn:dbarb adjoints}
\dbarbsp \vp &= \sum_{\atopp{I\in\I_{q-1}'}{J\in\I_q'}}\sum_{j=1}^{n-1} \ep^{jI}_J \Lbap_j \vp_J \ob_I
+ \sum_{\atopp{I\in\I_{q-1}'}{J\in\I_q'}} \overline{m^I_J} \vp_J  \ob_I \\
&= -\sum_{\atopp{I\in\I_{q-1}'}{J\in\I_q'}}\sum_{j=1}^{n-1}\ep^{jI}_J\big( L_j \vp_J - L_j\lp \vp_J + f_j \vp_J\big)\omb_I
+ \sum_{\atopp{I\in\I_{q-1}'}{J\in\I_q'}} \overline{m^I_J} \vp_J  \ob_I \nn
\end{align}
\begin{align*}
\dbarbsm \vp &= \sum_{\atopp{I\in\I_{q-1}'}{J\in\I_q'}}\sum_{j=1}^{n-1} \ep^{jI}_J \Lbam_j \vp_J \ob_I
+ \sum_{\atopp{I\in\I_{q-1}'}{J\in\I_q'}} \overline{m^I_J} \vp_J  \ob_I \\
&= -\sum_{\atopp{I\in\I_{q-1}'}{J\in\I_q'}}\sum_{j=1}^{n-1}\ep^{jI}_J\big( L_j \vp_J + L_j\lm \vp_J+ f_j \vp_J\big)\omb_I
+ \sum_{\atopp{I\in\I_{q-1}'}{J\in\I_q'}} \overline{m^I_J} \vp_J  \ob_I
\end{align*}
Consequently, we see that
\[
\dbarbsp = \dbarbs - [\dbarbs,\lp] \text{  and  } \dbarbsm = \dbarbs + [\dbarbs,\lm],
\]
and all three adjoints have the same domain. Finally, let $\dbarbspm$ be the adjoint of $\dbarb$ with respect to
$\la\cdot,\cdot\rapm$. 

The computations proving Lemma 4.8 and Lemma 4.9 and equation (4.4) in \cite{Nic06} 
can be applied here with only
a change of notation, so we have the following two
results, recorded here as Lemma \ref{lem: dbarbspm computation} and Lemma \ref{lem:energy form estimate for dbarbspm}. 
The meaning of the results is that $\dbarbspm$ acts like $\dbarbsp$ for forms whose support is basically $\Cp$
and $\dbarbsm$ on forms whose support is basically $\Cm$.
\begin{lem}\label{lem: dbarbspm computation}
On smooth $(0,q)$-forms,
\begin{multline*}
\dbarbspm = \dbarbs - \summ \zm^2 \tPsplm[\dbarbs,\lp] + \summ\zm^2 \tPsmlm [\dbarbs,\lm] \\
+ \summ \Big( \tzm [\tzm\Psplm\zm,\dbarb]^*\tzm \Psplm\zm + \zm\Psplam\tzm[\dbarbsp,\tzm\Psplm\zm]\tzm \\
+ \tzm [\tzm\Psmlm\zm,\dbarb]^*\tzm \Psmlm\zm + \zm\Psplam\tzm[\dbarbsm,\tzm\Psmlm\zm]\tzm + E_A \Big),
\end{multline*}
where the error term $E_A$ is a sum of order zero terms and lower order terms. Also, the symbol of $E_A$ is supported in
$\Com$ for each $\mu$.
\end{lem}

We are now ready to define the energy forms that we use. Let
\begin{align*}
\Qbpm(\phi,\vp) &= \la\dbarb\phi,\dbarb\vp\rapm + \la\dbarbspm\phi,\dbarbspm\vp\rapm \\
\Qbp(\phi,\vp)&= (\dbarb\phi,\dbarb\vp)_\lp + (\dbarbsp\phi,\dbarbsp\vp)_\lp \\
\Qbo(\phi,\vp) &= (\dbarb\phi,\dbarb\vp)_0 + (\dbarbs\phi,\dbarbs\vp)_0 \\
\Qbm(\phi,\vp) &= (\dbarb\phi,\dbarb\vp)_\lm + (\dbarbsm\phi,\dbarbsm\vp)_\lm. 
\end{align*}

\begin{lem}\label{lem:energy form estimate for dbarbspm} 
If $\vp$ is a smooth $(0,q)$-form on $M$, then there exist constants $K, K_\pm, K'$ with $K\geq 1$ so that
\begin{multline}\label{eqn:energy form -- dbarb commuted by psi-do}
K\Qbpm(\vp,\vp) + K_\pm\sumn \|\tzn\tPsoln\zn\vpn\|_0^2 +  K'\|\vp\|_0^2 + O_{\pm}(\|\vp\|_{-1}^2)
\geq \sumn \Big[ \Qbp(\tzn\Pspln\zn\vpn,\tzn\Pspln\zn\vpn) \\
+ \Qbo(\tzn\Psoln\zn\vpn,\tzn\Psoln\zn\vpn) + \Qbm(\tzn\Psmln\zn\vpn,\tzn\Psmln\zn\vpn) \Big]
\end{multline}
$K$ and $K'$ do not depend on $A$.
\end{lem}

Many of the subsequent proofs make use of the``lc/sc" argument: $-\ep\|x\|^2 -\ep^{-1}\|y\|^2 \leq
2\Rre((x,y)) \leq \ep \|x\|^2 + \ep^{-1}\|y\|^2$ where $(\cdot,\cdot)$ is any Hermitian inner product with associated
norm $\|\cdot\|$. Also, since that $\dbarbsp = \dbarb + lower order$, commuting $\dbarbsp$ by  $\Pspln$ creates error
terms of order 0 that do not depend on $\lp$ and lower order terms that may depend on $\lp$.

%
%
\section{The Basic Estimate}\label{sec:basic estimate}
The goal of this section is to prove a basic estimate for smooth forms on $M$. 
\begin{prop}\label{prop:basic estimate}
Let $M\subset \C^N$ be a compact, orientable, weakly pseudoconvex CR-manifold of dimension $n\geq 5$ 
and $1\leq q\leq n-2$. Assume that $M$
admits  functions $\lp$ and $\lm$ where
$\lp$ is strictly CR-plurisubharmonic on $(0,q)$-forms 
and $\lm$ is strictly CR-plurisubharmonic on $(0,n-1-q)$-forms
Let $\vp\in\Dom(\dbarb)\cap\Dom(\dbarbs)$. 
There exist constants $K$,  $K_\pm$, and $K_\pm'$ where $K$  does
not depend on  $\lp$ and $\lm$ (and consequently $A$) so that
\[
A_\pm \norm \vp\normpm^2 \leq K \Qbpm(\vp,\vp) + K \norm \vp\normpm^2 +
K_\pm\sumn\sum_{J\in\I_q'}\|\tzn\tPsoln\zn\vpn_J\|_0^2 + K_\pm' \|\vp\|_{-1}^2.
\]
The constant $A_\pm>0$ is the minimum of the CR-plurisubharmonicity constants $A_{\lp}$ and $A_{\lm}$.
\end{prop}

The proof of Proposition \ref{prop:basic estimate}  comes as the culmination of a series of calculations that started with
Lemma \ref{lem:energy form estimate for dbarbspm}.

\subsection{Local Estimates}
We work on a fixed $U = U_\nu$. On this neighborhood, as above, there exists an orthonormal basis of vector fields
$L_1,\dots,L_{n}$, $\Lb_1,\dots,\Lb_n$ so that
\begin{equation}\label{eqn:L, Lb adjoint 1}
[L_j,\Lb_k] = c_{jk}T + \sum_{\ell=1}^{n-1} (d_{jk}^\ell L_\ell - \bar d_{kj}^\ell \Lb_\ell) 
\end{equation}
if $1\leq j,k\leq n-1$, and $T = L_n-\Lb_n$, and for some fixed point $P$,
\[
[L_j,\Lb_k]\big|_P = c_{jk}T.
\]
Note that $c_{jk}$ are the coefficients of the Levi form.
Recall that $\Lbap$, $\Lba$, and $\Lbam$ are the adjoints of $\Lb$ in $(\cdot,\cdot)_\lp$, $(\cdot,\cdot)_0$,
and $(\cdot,\cdot)_\lm$, respectively. From \eqref{eqn:dbarb adjoints}, we see that 
\[
\Lbap_j = -L_j+L_j(\lp) -f_j\quad\text{and}\quad \Lbam_j = -L_j-L_j(\lm) -f_j,
\]
and plugging this into (\ref{eqn:L, Lb adjoint 1}),  we have
\begin{align*}
[\Lbap_j,\Lb_k] &= -c_{jk}T - \sum_{\ell=1}^{n-1} \Big(d_{jk}^\ell L_\ell - \bar d_{kj}^\ell \Lb_\ell\Big) 
-\Lb_k L_j\lp +\Lb_kf_j \\
[\Lbam_j,\Lb_k] &= -c_{jk}T - \sum_{k=1}^{n-1} \Big(d_{jk}^\ell L_\ell - \bar d_{kj}^\ell \Lb_k \Big) 
+\Lb_k L_j\lm + \Lb_k f_j
\end{align*}

For the inner product $\Qbpp$, we have the following estimate. 
\begin{lem}\label{lem: controlling Qbp}
Let $\vp$ be a $(0,q)$-form supported in $U$, $\vp\in\Dom(\dbarb)\cap\Dom(\dbarbs)$. There exists
$0< \ep' \ll 1$ so that
\begin{align*}
&\Qbpp \geq (1-\ep') \sum_{J\in\I_q'}\sum_{j=1}^{n-1} \|\Lb_j\vp_J\|_\lp^2
+\sum_{J\in\I_q'}\sum_{j\in J}\bigg[ \Rre\big\{ (c_{jj}T\vp_J,\vp_{J})_{\lp} \big\} 
\\ &+ \frac 12\big( (\Lb_j L_j(\lp)+L_j\Lb_j(\lp))\vp_J,\vp_{J})_\lp \big)
+ \frac 12\sum_{\ell=1}^{n-1}\big( (d_{jj}^\ell L_\ell(\lp) + \bar d_{jj}^\ell \bar L_\ell(\lp))\vp_J,\vp_{J}\big)_\lp \bigg]\\
&- \sum_{J,J'\in \I_q'}\sum_{\atopp{1\leq j,k\leq n-1}{j\neq k}} \ep^{kJ}_{jJ'}\bigg[ \Rre\big\{ (c_{jk}T\vp_J,\vp_{J'})_{\lp} \big\} 
+ \frac 12\big( (\Lb_k L_j(\lp)+L_j\Lb_k(\lp))\vp_J,\vp_{J'})_\lp \big) \\
&+ \frac 12\sum_{\ell=1}^{n-1}\big( (d_{jk}^\ell L_\ell(\lp) + \bar d_{kj}^\ell \bar L_\ell(\lp))\vp_J,\vp_{J'}\big)_\lp
 \bigg] + O(\|\vp\|_0^2).
\end{align*}
\end{lem}

\begin{proof}First, observe 
\[
(\dbarbsp\vp,\dbarbsp\vp)_\lp = \sum_{\atopp{I\in \I_{q-1}'}{J,J'\in\I_q'}}
\sum_{\atopp{1\leq j,k\leq n-1}{j\neq k}} \ep^{jI}_J \ep^{kI}_{J'}(\Lbap_j\vp_J,\Lbap_k \vp_{J'})_\lp 
+ O\big(\|\vp\|_{\lp}^2 + (\sum_{j=1}^{n-1}\|\Lb_j\vp\|_{\lp}^2)^{1/2}\|\vp\|_{\lp}\big).
\]
However, if $j\neq k$, then
$\ep^{jI}_J \ep^{kI}_{J'} = \ep^{kjI}_{kJ} \ep^{jkI}_{jJ'} = -\ep^{jkI}_{kJ} \ep^{jkI}_{jJ'} =
-\ep^{kJ}_{jJ'}$.
Consequently,
\begin{align*}
\|\dbarbsp\vp\|_{\lp}^2
=& \sum_{J\in\I_q'}\sum_{j\in J} \|\Lbap_j\vp_J\|_{\lp}^2 
- \sum_{J,J'\in\I_q'}\sum_{\atopp{1\leq j,k\leq n-1}{j\neq k}}\ep^{kJ}_{jJ'} (\Lbap_j\vp_J,\Lbap_k \vp_{J'})_\lp  \\
&+ O\big(\|\vp\|_{\lp}^2 + (\sum_{j=1}^{n-1}\|\Lb_j\vp\|_{\lp}^2)^{1/2}\|\vp\|_{\lp}\big) \\
=& \sum_{J\in\I_q'}\sum_{j\in J} \|\Lb_j\vp_J\|_{\lp}^2 + \sum_{J\in\I_q'}\sum_{j\in J} \big([\Lb_j,\Lbap_j]\vp_J,\vp_J\big)_\lp\\
&- \sum_{J,J'\in\I_q'}\sum_{\atopp{1\leq j,k\leq n-1}{j\neq k}}\ep^{kJ}_{jJ'} (\Lbap_j\vp_J,\Lbap_k \vp_{J'})_\lp
+ O\big(\|\vp\|_{\lp}^2 + (\sum_{j=1}^{n-1}\|\Lb_j\vp\|_{\lp}^2)^{1/2}\|\vp\|_{\lp}\big)
\end{align*}

Second, from the calculation of $\dbarb$ above, we compute
\begin{align*}
&\|\dbarb\vp\|_\lp^2 
= \sum_{\atopp{J,J'\in\I_q'}{K\in\I_{q+1}'}}\sum_{\atopp{1\leq j,k\leq n-1}{j\neq k}} \ep^{kJ}_K\ep^{jJ'}_K
(\Lb_k \vp_J,\Lb_j\vp_{J'})_\lp 
O\big(\|\vp\|_{\lp}^2 + (\sum_{j=1}^{n-1}\|\Lb_j\vp\|_{\lp}^2)^{1/2}\|\vp\|_{\lp}\big) \\
&= \sum_{J\in\I_q}\sum_{j\not\in J} \|\Lb_j \vp_J\|_\lp^2 + \sum_{J,J'\in\I_q}\sum_{\atopp{1\leq j,k\leq n-1}{j\neq k}} \ep^{kJ}_{jJ'}
(\Lb_k \vp_J,\Lb_j \vp_{J'})_\lp 
+ O\big(\|\vp\|_{\lp}^2 + (\sum_{j=1}^{n-1}\|\Lb_j\vp\|_{\lp}^2)^{1/2}\|\vp\|_{\lp}\big) \\
&= \sum_{J\in\I_q}\sum_{j\not\in J} \|\Lb_j \vp_J\|_\lp^2 + \sum_{J,J'\in\I_q}\sum_{\atopp{1\leq j,k\leq n-1}{j\neq k}} \ep^{kJ}_{jJ'}
(\Lbap_j \vp_J,\Lbap_k \vp_{J'})_\lp \\
&+ \sum_{J,J'\in\I_q}\sum_{\atopp{1\leq j,k\leq n-1}{j\neq k}} \ep^{kJ}_{jJ'} \big( [\Lbap_j,\Lb_k] \vp_J,\vp_{J'}\big)_\lp
+ O\big(\|\vp\|_{\lp}^2 + (\sum_{j=1}^{n-1}\|\Lb_j\vp\|_{\lp}^2)^{1/2}\|\vp\|_{\lp}\big).
\end{align*}
By a lc/sc argument, 
\[
\big(\sum_{j=1}^{n-1}\|\Lb_j\vp\|_{\lp}^2\big)^{1/2}\|\vp\|_{\lp}
\geq -\ep\sum_{j=1}^{n-1}\|\Lb_j\vp\|_{\lp}^2 - \frac 1\ep \|\vp\|_{\lp}^2,
\]
so adding together our computations yields
\begin{multline}\label{eqn:Qbpp first est}
\Qbpp \geq (1-\ep)\sum_{J\in\I_q'}\sum_{j=1}^{n-1} \|\Lb_j\vp_J\|_\lp^2
+ \sum_{J\in\I_q}\sum_{j\in J} \big([\Lb_j,\Lbap_j] \vp_J,\vp_J\big)_\lp \\
+\sum_{J,J'\in\I_q'}\sum_{\atopp{1\leq j,k\leq n-1}{j\neq k}} \ep^{kJ}_{jJ'} \big([\Lbap_j,\Lb_k]\vp_J,\vp_{J'}\big)_\lp + O(\|\vp\|_{\lp}^2).
\end{multline}

Recall that the commutator
\[
[\Lbap_j,\Lb_k] = -c_{jk}T - \sum_{\ell=1}^{n-1} \Big(d_{jk}^\ell L_\ell - \bar d_{kj}^\ell \Lb_\ell\Big) 
-\Lb_k L_j\lp +\Lb_kf_j ,
\]
and note that
\[
\big|\big( \bar d_{kj}^\ell \Lb_\ell\vp_J,\vp_{J'}\big)_\lp\big|  \leq 
\ep \|\Lb_\ell\vp_J\|_\lp^2 + C_\ep \|\vp\|_{\lp}^2.
\]
Consequently,
\begin{multline*}
\Qbpp \geq (1-\ep)\sum_{J\in\I_q'}\sum_{j=1}^{n-1} \|\Lb_j\vp_J\|_\lp^2 \\
+ \Rre\bigg\{ \sum_{J\in\I_q'}\sum_{j\in J}\bigg[ \big( c_{jj} T\vp_J,\vp_J\big)_\lp 
+ \sum_{\ell=1}^{n-1} \big( d^\ell_{jj} L_\ell \vp_J,\vp_J\big)_\lp 
+ \big(\Lb_j\Lb_j\lp \vp_J,\vp_J\big)_\lp\bigg] \bigg\} \\
- \Rre\bigg\{ \sum_{J,J'\in\I_q'}\sum_{\atopp{1\leq j,k\leq n-1}{j\neq k}} \ep^{kJ}_{jJ'}\bigg[ \big( c_{jk} T\vp_J,\vp_{J'}\big)_{\lp}  
+ \sum_{\ell=1}^{n-1} \big(d_{jk}^\ell L_\ell \vp_J,\vp_{J'}\big)_\lp 
+ \big(\Lb_k L_j\lp\vp_J,\vp_{J'}\big)_{\lp}\bigg]\bigg\}
+ O(\|\vp\|_0^2).
\end{multline*} 
Also,
\begin{multline*}
\ep^{kJ}_{jJ'} \Rre\big\{ \big(d_{jk}^\ell L_\ell \vp_J,\vp_{J'}\big)_\lp \big\}
= \ep^{kJ}_{jJ'} \Rre\big\{ \big(L_\ell (d_{jk}^\ell  \vp_J),\vp_{J'}\big)_\lp  
- \ep^{kJ}_{jJ'} \Rre\{ \big( L_\ell(d_{jk}^\ell)  \vp_J,\vp_{J'}\big)_\lp \big\} \\
= \ep^{kJ}_{jJ'} \Rre\big\{ \big(-\Lbap_\ell(d^\ell_{jk}\vp_J),\vp_{J'}\big)_{\lp} 
+ \big(d^\ell_{jk}L_\ell(\lp)\vp_J,\vp_{J'}\big)_\lp \big\}
+ O(\|\vp\|_{\lp}^2)\\
\geq -\ep \|\Lb_\ell\vp_{J'}\|_{\lp}^2 + \ep^{kJ}_{jJ'} \Rre\big\{ \big(d^\ell_{jk}L_\ell(\lp)\vp_j,\vp_k\big)_\lp \big\}
+O(\|\vp\|_{\lp}^2).
\end{multline*}
Recalling that $\Rre z = \Rre\z$ for any complex number $z$, we have
\begin{align*}
&\sum_{J,J'\in\I_q} \sum_{j,k,\ell=1}^{n-1}\ep^{kJ}_{jJ'}\Rre\Big\{ \big(d^\ell_{jk}L_\ell(\lp)\vp_J,\vp_{J'}\big)_\lp\Big\}\\
&= \frac {1}2  \sum_{J,J'\in\I_q'}
\sum_{j,k,\ell=1}^{n-1}\ep^{kJ}_{jJ'}\Rre\Big\{ \big(d^\ell_{jk}L_\ell(\lp)\vp_J,\vp_{J'}\big)_\lp
+  \big(d^\ell_{kj}L_\ell(\lp)\vp_{J'},\vp_{J}\big)_\lp\Big\}\\
&= \frac {1}2  \sum_{J,J'\in\I_q'}
\sum_{j,k,\ell=1}^{n-1}\ep^{kJ}_{jJ'}\Rre\Big\{ \big(d^\ell_{jk}L_\ell(\lp)\vp_J,\vp_{J'}\big)_\lp
+  \big(\vp_{J'},\bar d^\ell_{kj}\Lb_\ell(\lp)\vp_{J}\big)_\lp\Big\}\\
&= \frac {1}2  \sum_{J,J'\in\I_q'}
\sum_{j,k,\ell=1}^{n-1}\ep^{kJ}_{jJ'}\Rre\Big\{ \big(d^\ell_{jk}L_\ell(\lp)\vp_J,\vp_{J'}\big)_\lp
+ \overline{\big(\bar d^\ell_{kj}\Lb_\ell(\lp)\vp_{J}, \vp_{J'}\big)_\lp}\Big\}\\
&= \frac {1}2 \sum_{J,J'\in\I_q'}
\sum_{j,k,\ell=1}^{n-1}\ep^{kJ}_{jJ'} \Big(\big(d^\ell_{jk}L_\ell(\lp)\vp_J,\vp_{J'}\big)_\lp
+ \big(\bar d^\ell_{kj}\Lb_\ell(\lp)\vp_{J}, \vp_{J'}\big)_\lp\Big)\\
\end{align*}
Similarly,
\begin{multline*}
\Rre\bigg\{ \sum_{J,J'\in\I_q'}\sum_{\atopp{1\leq j,k\leq n-1}{j\neq k}} \ep^{kJ}_{jJ'} \big(\Lb_k L_j\lp\vp_J,\vp_{J'}\big)_{\lp} \bigg\} \\
= \frac {1}2 \sum_{J,J'\in\I_q'}\sum_{\atopp{1\leq j,k\leq n-1}{j\neq k}} \ep^{kJ}_{jJ'}\Big( \big(\Lb_k L_j\lp\vp_J,\vp_{J'}\big)_{\lp}
+ \big(L_j \Lb_k \lp\vp_J,\vp_{J'}\big)_{\lp} \Big)
\end{multline*}
\end{proof}

Next, we concentrate on the $\Qbmp$ term. 
\begin{lem}\label{lem:controlling Qbm}
Let $\vp$ be a $(0,q)$-form supported in $U$, $\vp\in\Dom(\dbarb)\cap\Dom(\dbarbs)$. There exists
$0< \ep' \ll 1$ so that
\begin{align*}
&\Qbmp \geq (1-\ep') \sum_{J\in\I_q'}\sum_{j=1}^{n-1}\Big[ \|\Lbam_j\vp_J\|_\lm^2
+\sum_{J\in\I_q'}\sum_{j\in J}\bigg[ \Rre\big\{ -(c_{jj}T\vp_J,\vp_{J})_{\lm} \big\} 
\\ &+ \frac 12\big( (\Lb_j L_j(\lm)+L_j\Lb_j(\lm))\vp_J,\vp_{J})_\lm \big)
+ \frac 12\sum_{\ell=1}^{n-1}\big( (d_{jj}^\ell L_\ell(\lm) + \bar d_{jj}^\ell \bar L_\ell(\lm))\vp_J,\vp_{J}\big)_\lm \bigg]\\
&+ \sum_{J,J'\in \I_q'}\sum_{\atopp{1\leq j,k\leq n-1}{j\neq k}} \ep^{kJ}_{jJ'}\bigg[ \Rre\big\{ -(c_{jk}T\vp_J,\vp_{J'})_{\lm} \big\} 
+ \frac 12\big( (\Lb_k L_j(\lm)+L_j\Lb_k(\lm))\vp_J,\vp_{J'})_\lm \big) \\
&+ \frac 12\sum_{\ell=1}^{n-1}\big( (d_{jk}^\ell L_\ell(\lm) + \bar d_{kj}^\ell \bar L_\ell(\lm))\vp_J,\vp_{J'}\big)_\lm
 \bigg] + O(\|\vp\|_0^2).
\end{align*}
\end{lem}

\begin{proof}
This lemma is proved with the same techniques as the previous lemma. By the argument leading up to
\eqref{eqn:Qbpp first est}, we have
\begin{align*}
\Qbmp &= \sum_{J\in\I_q'}\sum_{j\not\in J} \|\Lb_j \vp_J\|_\lm^2 
+ \sum_{J,J'\in\I_q}\sum_{\atopp{1\leq j,k\leq n-1}{j\neq k}} \ep^{kJ}_{jJ'}
(\Lbam_j \vp_J,\Lbam_k \vp_{J'})_\lm \\
&+ \sum_{J,J'\in\I_q}\sum_{\atopp{1\leq j,k\leq n-1}{j\neq k}} \ep^{kJ}_{jJ'} \big( [\Lbam_j,\Lb_k] \vp_J,\vp_{J'}\big)_\lm\\
&+ \sum_{J\in\I_q'}\sum_{j\in J} \|\Lbam_j\vp_J\|_{\lm}^2 
- \sum_{J,J'\in\I_q'}\sum_{\atopp{1\leq j,k\leq n-1}{j\neq k}}\ep^{kJ}_{jJ'} (\Lbam_j\vp_J,\Lbam_k \vp_{J'})_\lm  \\
&+ O\big(\|\vp\|_{\lm}^2 + (\sum_{J\in\I_q'}\sum_{j=1}^{n-1}\|\Lbam_j\vp_J\|_{\lm}^2)^{1/2}\|\vp\|_{\lm}\big)
\end{align*}
By integration by parts,
\[
\|\Lb_j\vp_J\|_{\lm}^2 = \|\Lbam_j\vp_J\|_\lm^2 + \big( [\Lbam_j,\Lb_j]\vp_J,\vp_J\big)_\lm.
\]
Thus,
\begin{align*}
&\Qbmp = \sum_{J\in\I_q'}\sum_{j=1}^{n-1}\|\Lbam_j \vp_J\|_\lm^2 
+ \sum_{J\in\I_q'}\sum_{j\not\in J}\big( [\Lbam_j,\Lb_j]\vp_J,\vp_J\big)_\lm \\
&+ \sum_{J,J'\in\I_q}\sum_{\atopp{1\leq j,k\leq n-1}{j\neq k}} \ep^{kJ}_{jJ'} \big( [\Lbam_j,\Lb_k] \vp_J,\vp_{J'}\big)_\lm
+O\big(\|\vp\|_{\lm}^2 + (\sum_{J\in\I_q'}\sum_{j=1}^{n-1}\|\Lbam_j\vp_J\|_{\lm}^2)^{1/2}\|\vp\|_{\lm}\big).
\end{align*}
Following the argument of Lemma \ref{lem: controlling Qbp}, we proceed as above.
\end{proof}

The significance of  the estimates in Lemma \ref{lem: controlling Qbp} and Lemma \ref{lem:controlling Qbm} is
demonstrated by the 
multilinear algebra in Appendix \ref{app:multilinear algebra}, and it highlights the need for \CRPq as well
as \CRPnq.

We need the following versions of the sharp G{\aa}rding inequality. This is Theorem 7.1 in \cite{Nic06} written
for forms. It can be proved by following proofs of Theorem 3.1 and Theorem 3.2 in \cite{LaNi66} line by line (making
the obvious modifications). 
\begin{thm}\label{thm: Garding}
If $P = (p_{jk}(z,D))$ is a matrix first order pseudodifferential operator. If $p(z,\xi)$ is Hermitian and
the sum of any collection of $q$ eigenvalues is nonnegative,
then there exists a constant $C>0$ so that for any $(0,q)$-form $u$,
\[
 \Rre\Big\{\sum_{J\in\I_q'} \big(p_{jj}(\cdot,D) u_J, u_J\big) 
- \sum_{J,J'\in\I_q'} \sum_{\atopp{1\leq j,k\leq m}{j\neq k}}
 \ep^{kJ}_{jJ'} \big( p_{jk}(\cdot,D) u_J,u_{J'}\big) \Big \}
\geq -C \|u\|^2.
\]
If $p(z,\xi)$ is Hermitian and
the sum of any collection of $(n-1-q)$ eigenvalues is nonnegative, then 
\[
 \Rre\Big\{\sum_{J\in\I_q'} \big(p_{jj}(\cdot,D) u_J, u_J\big) 
+ \sum_{J,J'\in\I_q'} \sum_{\atopp{1\leq j,k\leq m}{j\neq k}}
 \ep^{kJ}_{jJ'} \big( p_{jk}(\cdot,D) u_J,u_{J'}\big) \Big \}
\geq -C \|u\|^2.
\]
\end{thm}

\begin{cor}\label{cor:Garding}
Let $R$ be a first order pseudodifferential operator such that  $\sigma(R)\geq\kappa$ where $\kappa$ is some positive
constant and $(h_{jk})$ a hermitian matrix (that does not depend on $\xi$).
Then there exists a constant $C$ such that if the sum of any $q$ eigenvalue of $(h_{jk})$ is nonnegative, then
\begin{multline*}
\Rre\Big\{\sum_{J\in\I_q'} \big(h_{jj}R u_J, u_J\big) 
- \sum_{J,J'\in\I_q'} \sum_{\atopp{1\leq j,k\leq m}{j\neq k}}
 \ep^{kJ}_{jJ'} \big( h_{jk}R u_J,u_{J'}\big) \Big \} \\
\geq \kappa  \Rre\Big\{\sum_{J\in\I_q'} \big((h_{jj} u_J, u_J\big) 
- \sum_{J,J'\in\I_q'} \sum_{\atopp{1\leq j,k\leq m}{j\neq k}}
 \ep^{kJ}_{jJ'} \big( h_{jk} u_J,u_{J'}\big) \Big \} -C \|u\|^2 .
\end{multline*}
and if the
the sum of any collection of $(n-1-q)$ eigenvalues of $(h_{jk})$ is nonnegative, then
\begin{multline*}
\sum_{jk} \Rre\Big\{\sum_{J\in\I_q'} \big((h_{jj}R u_J, u_J\big) 
+ \sum_{J,J'\in\I_q'} \sum_{\atopp{1\leq j,k\leq m}{j\neq k}}
 \ep^{kJ}_{jJ'} \big( h_{jk}R u_J,u_{J'}\big) \Big \} \\
\geq \kappa \sum_{jk} \Rre\Big\{\sum_{J\in\I_q'} \big((h_{jj} u_J, u_J\big) 
+ \sum_{J,J'\in\I_q'} \sum_{\atopp{1\leq j,k\leq m}{j\neq k}}
 \ep^{kJ}_{jJ'} \big( h_{jk} u_J,u_{J'}\big) \Big \} -C \|u\|^2 .
\end{multline*}
\end{cor}
Note that $(h_{jk})$ may be a matrix-valued function in $z$ but may not depend on $\xi$.
\begin{proof} Apply the previous theorem with $P$ where $p_{jk} = h_{jk}(R-\kappa)$.
\end{proof}

We need G{\aa}rding's inequality to prove the following analog to Lemma 4.12 in  \cite{Nic06}.
\begin{lem} \label{lem:T bound lp}
Let $M$ be a weakly pseudoconvex CR-manifold and $\vp$ a $(0,q)$-form supported on $U'$ so that
up to a smooth term $\hat\vp$ is supported in $\Cp$. Then 
\begin{multline*}
\Rre\Big\{\sum_{J\in\I_q'} \big(c_{jj}T \vp_J, \vp_J\big)_\lp 
- \sum_{J,J'\in\I_q'} \sum_{\atopp{1\leq j,k\leq m}{j\neq k}}
 \ep^{kJ}_{jJ'} \big( c_{jk}T \vp_J,\vp_{J'}\big)_\lp \Big \} \\
\geq A  \Rre\Big\{\sum_{J\in\I_q'} \big((c_{jj} \vp_J, \vp_J\big)_\lp 
- \sum_{J,J'\in\I_q'} \sum_{\atopp{1\leq j,k\leq m}{j\neq k}}
 \ep^{kJ}_{jJ'} \big( c_{jk} \vp_J,\vp_{J'}\big)_\lp \Big \} 
+ O(\|\vp\|_{\lp}^2) + O_A (\|\tzn\tPsol\vp\|_0^2).
\end{multline*}
where the constant in $ O(\|\vp\|_{\lp}^2)$ does not depend on $A$.
\end{lem}

\begin{proof} Let $\tPspl$ be a pseudodifferential operator of order zero whose symbol dominates
$\hat\phi$ (up to a smooth error) and is supported in $\tCp$.
By the support conditions of $\vp$ and $\hat\vp$,
\begin{align*}
\sum_{J\in\I_q'}& \big(c_{jj}T \vp_J, \vp_J\big)_\lp 
- \sum_{J,J'\in\I_q'} \sum_{\atopp{1\leq j,k\leq m}{j\neq k}}
 \ep^{kJ}_{jJ'} \big( c_{jk}T \vp_J,\vp_{J'}\big)_\lp \\
=& \sum_{J\in\I_q'} \big(c_{jj}T \vp_J, (\tPspla\tPspl + (Id- \tPspla\tPspl))\vp_J\big)_\lp  \\
&- \sum_{J,J'\in\I_q'} \sum_{\atopp{1\leq j,k\leq m}{j\neq k}} 
 \ep^{kJ}_{jJ'} \big( c_{jk}T \vp_J,(\tPspla\tPspl + (Id- \tPspla\tPspl))\vp_{J'}\big)_\lp\\
=& \sum_{J\in\I_q'} \big(c_{jj}T \vp_J, \tPspla\tPspl\vp_J\big)_\lp  \\
&- \sum_{J,J'\in\I_q'} \sum_{\atopp{1\leq j,k\leq m}{j\neq k}} 
 \ep^{kJ}_{jJ'} \big( c_{jk}T \vp_J,\tPspla\tPspl\vp_{J'}\big)_\lp  + smoother\ terms \\
=& \sum_{J\in\I_q'} \big( \tz e^{-\lp}c_{jj}  \tPspl T \vp_J, \tz\tPspl\vp_J\big)_0  \\
&- \sum_{J,J'\in\I_q'} \sum_{\atopp{1\leq j,k\leq m}{j\neq k}} 
 \ep^{kJ}_{jJ'} \big( \tz e^{-\lp} c_{jk} \tPspl T \vp_J,\tz\tPspl\vp_{J'}\big)_0  + smoother\ terms \\
=& \sum_{J\in\I_q'} \big( \tz\tPspla\tz^2 e^{-\lp}c_{jj}  \tPspl T \vp_J, \vp_J\big)_0  \\
&- \sum_{J,J'\in\I_q'} \sum_{\atopp{1\leq j,k\leq m}{j\neq k}} 
 \ep^{kJ}_{jJ'} \big( \tz\tPspla\tz^2 e^{-\lp} c_{jk} \tPspl T \vp_J,\vp_{J'}\big)_0  + smoother\ terms \\
&= \sum_{j,k=1}^{n-1} (\tz\tPspla\tz^2e^{-\lp}c_{jk}T\tPspl\vp_j,\vp_k)_0 + smoother\ terms.
\end{align*}
where smoother terms are $O(\|\vp\|_{-1}^2)$ or better (and the constant may depend on $A$). One fact quickly computed
and used implicitly above is that $\sigma(\tPspla T) = \sigma(T\tPspla) = \xi_{2n-1}\tpspl(\xi)$ 
(up to smooth terms) when applied to $\vp$. Next, we will compute
$\sigma(\tPspla\tzn^2e^{-\lp}c_{jk})$. $\sigma(\tPspl)\equiv 1$ on $\Cp$, so $\sigma(\tPspla)\equiv 1$ on $\Cp$ as well,
and it follows that $\sigma(\tPspla) = \tpspl(\xi)$ up to terms supported in $\Co\setminus\Cp$. Thus, up to errors
on $\Co\setminus\Cp$,
\[
\sigma(\tPspla\tzn^2e^{-\lp}c_{jk}) = \sum_{\beta\geq 0}\frac{1}{\beta!}\p^\beta_\xi \tpspl(\xi) D^\beta_x(\tzn^2e^{-\lp}c_{jk}) 
= \tpspl(\xi)\tzn^2e^{-\lp}c_{jk},
\]
and on $\Cp$,
\begin{multline*}
\sigma(\tPspla\tzn^2e^{-\lp}c_{jk} T\Pspl) =\sum_{\alpha}\frac{1}{\alpha!} \p^\alpha_\xi\sigma(\tPspla\tzn^2e^{-\lp}c_{jk})
D^{\alpha}_x\sigma(T\tPspl)\\
= \sum_{\alpha}\frac{1}{\alpha!} \p^\alpha_\xi (\tpspl(\xi)\tzn^2e^{-\lp}c_{jk})
D^{\alpha}_x\sigma(\xi_{2n-1}\tpspl(\xi)) = \tzn^2e^{-\lp} c_{jk}\xi_{2n-1}.
\end{multline*}
By construction, $\xi_{2n-1}\geq A$ on $\Cp$ and $(\tzn e^{-\lp} c_{jk})$ is positive semi-definite (and hence the 
sum of any $q$ eigenvalues is nonnegative), so we can apply
Corollary \ref{cor:Garding} with $T$ as $R$ and $(e^{-\lp} c_{jk})$ as $(h_{jk})$ to conclude that there
exists a constant $C$ independent of $A$ so that
\begin{align*}
\sum_{J\in\I_q'}& \big(c_{jj}T \vp_J, \vp_J\big)_\lp 
- \sum_{J,J'\in\I_q'} \sum_{\atopp{1\leq j,k\leq m}{j\neq k}}
 \ep^{kJ}_{jJ'} \big( c_{jk}T \vp_J,\vp_{J'}\big)_\lp \\
 \geq & A\bigg( \sum_{J\in\I_q'} \big( \tz^2 e^{-\lp} c_{jj} \vp_J, \vp_J\big)_0
- \sum_{J,J'\in\I_q'} \sum_{\atopp{1\leq j,k\leq m}{j\neq k}}
 \ep^{kJ}_{jJ'} \big( \tz^2 e^{-\lp} c_{jk} \vp_J,\vp_{J'}\big)_0  \bigg) \\
&- C\|\vp\|_{\lp}^2 + O(\|\vp\|_{-1}^2) +
O_A(\|\tzn\tPsol\vp\|_0^2\\
=& A \sum_{J\in\I_q'} \big(c_{jj} \vp_J, \vp_J\big)_\lp 
- \sum_{J,J'\in\I_q'} \sum_{\atopp{1\leq j,k\leq m}{j\neq k}}
 \ep^{kJ}_{jJ'} \big( c_{jk} \vp_J,\vp_{J'}\big)_\lp  +  O(\|\vp\|_{0}^2) +
O_A(\|\tz\tPsol\vp\|_0^2).
\end{align*}
\end{proof}

By the same argument, we have the following:
\begin{lem}\label{lem:T bound lm}
Let $\vp$ be a $(0,q)$-form supported on $U$ so that up to a smooth term, $\hat\vp$ is supported in $\Cm$, then
\begin{multline*}
\sum_{J\in\I_q'} \big(c_{jj}(-T) \vp_J, \vp_J\big)_\lm 
+ \sum_{J,J'\in\I_q'} \sum_{\atopp{1\leq j,k\leq m}{j\neq k}}
 \ep^{kJ}_{jJ'} \big( c_{jk}(-T) \vp_J,\vp_{J'}\big)_\lm \\
\geq  A\bigg(\sum_{J\in\I_q'} \big(c_{jj} \vp_J, \vp_J\big)_\lm 
+ \sum_{J,J'\in\I_q'} \sum_{\atopp{1\leq j,k\leq m}{j\neq k}}
 \ep^{kJ}_{jJ'} \big( c_{jk} \vp_J,\vp_{J'}\big)_\lm \bigg)
+ O(\|\vp\|_{\lm}^2) + O_A (\|\tzn\tPsol\vp\|_0^2).
\end{multline*}
\end{lem}


We now review the two local results from \cite{Nic06} that are crucial in proving the basic estimate
Proposition \ref{prop:basic estimate}. Let $(\sjkp)_{j,k=1}^{n-1}$ be the matrix defined by
\[
\sjkp = \frac 12\Big(\Lb_k L_j(\lp)+L_j\Lb_k(\lp)+ \sum_{\ell=1}^{n-1}(d_{jk}^\ell L_\ell(\lp)+\bar d_{kj}^\ell \Lb_\ell(\lp)) \Big) 
+ A_0 c_{jk}.
\]

\begin{prop} \label{prop:local results for Pspl}
Let $\vp\in\Dom(\dbarb)\cap\Dom(\dbarbs)$ be a $(0,q)$-form supported in $U$. Assume that
$\lp$ is a strictly CR-plurisubharmonic function on $(0,q)$-forms with CR-plurisubharmonicity constant $\Alp$.
Then there exists a constant $C$ that is independent of $\Alp$ so that
\[
\Qbp(\tz\Pspl\vp,\tz\Pspl\vp) + C\|\tz\Pspl\vp\|_\lp^2 + O_\lp(\|\tz\tPsol\vp\|_0^2) 
\geq A_\lp \|\tz\Pspl\vp\|_{\lp}^2.
\]
\end{prop}

\begin{proof}
Since $\vp\in \Dom(\dbarb)\cap\Dom(\dbarbs)$, it follows that $\tz\Pspl\vp\in \Dom(\dbarb)\cap\Dom(\dbarbs)$. Moreover,
$\supp(\tz\Pspl\vp)\subset U'$. 
By Lemma \ref{lem: controlling Qbp},
\begin{align*}
&\Qbp(\tz\Pspl\vp,\tz\Pspl\vp) \geq (1-\ep') \sum_{J\in\I_q'}\sum_{j=1}^{n-1} \|\Lb_j\tz\Pspl\vp_J\|_\lp^2 \\
& + \Rre\Big\{\sum_{J\in\I_q'}\sum_{j\in J}  (c_{jj}T\tz\Pspl\vp_J,\tz\Pspl\vp_{J})_{\lp} 
- \sum_{J,J'\in \I_q'}\sum_{\atopp{1\leq j,k\leq n-1}{j\neq k}} \ep^{kJ}_{jJ'}
(c_{jk}T\tz\Pspl\vp_J,\tz\Pspl\vp_{J'})_{\lp} \Big\} \\ 
\\ & + \frac12 \sum_{J\in\I_q'}\sum_{j\in J}\Big[ 
\big( (\Lb_j L_j(\lp)+L_j\Lb_j(\lp))\tz\Pspl\vp_J,\tz\Pspl\vp_{J})_\lp \big)\\
&\hspace{2in} + \sum_{\ell=1}^{n-1}\big( (d_{jj}^\ell L_\ell(\lp) 
+ \bar d_{jj}^\ell \bar L_\ell(\lp))\tz\Pspl\vp_J,\tz\Pspl\vp_{J}\big)_\lp\Big] \\
&- \frac 12\sum_{J,J'\in \I_q'}\sum_{\atopp{1\leq j,k\leq n-1}{j\neq k}} \ep^{kJ}_{jJ'}
\bigg[\big( (\Lb_k L_j(\lp)+L_j\Lb_k(\lp))\tz\Pspl\vp_J,\tz\Pspl\vp_{J'})_\lp \big) \\
&\hspace{1in} + \sum_{\ell=1}^{n-1}\big( (d_{jk}^\ell L_\ell(\lp) 
+ \bar d_{kj}^\ell \bar L_\ell(\lp))\tz\Pspl\vp_J,\tz\Pspl\vp_{J'}\big)_\lp
 \bigg] + O(\|\tz\Pspl\vp\|_0^2).
\end{align*}
To control the $T$ terms, we use Lemma \ref{lem:T bound lp}  since $\supp\tz\subset U'$, and
the Fourier transform of $\tz\Pspl\vp$ is supported in $\Cp$ up to a smooth term. Indeed, with $A=A_0$ (and $A_0$ from
the definition of \CRPq), we have
\begin{multline*} 
\Rre\Big\{\sum_{J\in\I_q'}\sum_{j\in J}  (c_{jj}T\tz\Pspl\vp_J,\tz\Pspl\vp_{J})_{\lp} 
- \sum_{J,J'\in \I_q'}\sum_{\atopp{1\leq j,k\leq n-1}{j\neq k}} \ep^{kJ}_{jJ'}
(c_{jk}T\tz\Pspl\vp_J,\tz\Pspl\vp_{J'})_{\lp} \Big\}\\
\geq A_0 \Big[\sum_{J\in\I_q'}\sum_{j\in J}  (c_{jj}\tz\Pspl\vp_J,\tz\Pspl\vp_{J})_{\lp} 
- \sum_{J,J'\in \I_q'}\sum_{\atopp{1\leq j,k\leq n-1}{j\neq k}} \ep^{kJ}_{jJ'}
(c_{jk}\tz\Pspl\vp_J,\tz\Pspl\vp_{J'})_{\lp} \Big] \\
+ O(\|\tz\Pspl\vp\|_\lp^2) + O_\lp(\|\tz\tPsol\vp\|_0^2) 
\end{multline*}
Putting these estimates together, we have
\begin{multline*}
\Qbp(\tz\Pspl\vp,\tz\Pspl\vp) \geq \\ 
\sum_{J\in\I_q'}\sum_{j\in J}  (s_{jj}^+\tz\Pspl\vp_J,\tz\Pspl\vp_{J})_{\lp} 
- \sum_{J,J'\in \I_q'}\sum_{\atopp{1\leq j,k\leq n-1}{j\neq k}} \ep^{kJ}_{jJ'}
(\sjkp\tz\Pspl\vp_J,\tz\Pspl\vp_{J'})_{\lp}  \\
+ O(\|\tz\Pspl\vp\|_\lp^2) + O_\lp(\|\tz\tPsol\vp\|_0^2) ).
\end{multline*}
Recall that $\lp$ is strictly CR-plurisubharmonic on $(0,q)$-forms with 
CR-plurisubharmonicity constant $A_\lp$.
In local coordinates, if $L = \sum_{j=1}^{n-1} \xi_j L_j$, then
\[
\Big\la \frac12\big(\p_b\dbarb\lp - \dbarb\p_b\lp\big) + A_0 d\gamma, L\wedge\Lb\Big\ra
 = \sum_{j,k=1}^{n-1} \sjkp \xi_j\bar\xi_k,
\]
and $(\sjkp)$ is a Hermitian matrix. Therefore, by the multilinear algebra lemmas,
Lemma \ref{lem:multilinear algebra} and Lemma \ref{lem:multilin for q at a time},
\[
\Qbp(\tz\Pspl\vp,\tz\Pspl\vp) + C\|\tz\Pspl\vp\|_\lp^2 + O_\lp(\|\tz\tPsol\vp\|_0^2) 
\geq A_\lp \|\tz\Pspl\vp\|_{\lp}^2.
\]
where the constant $C$ is independent of $A_\lp$.
\end{proof}

Let
\[
\sjkm = \frac 12\Big(\Lb_k L_j(\lm)+L_j\Lb_k(\lm)+ \sum_{\ell=1}^{n-1}(d_{jk}^\ell L_\ell(\lm)+\bar d_{kj}^\ell \Lb_\ell(\lm)) \Big) 
+ A_0 c_{jk}.
\] 

\begin{prop} \label{prop:local results for Psml}
Let $\vp\in\Dom(\dbarb)\cap\Dom(\dbarbs)$ be a $(0,q)$-form supported in $U$. Assume that
$\lm$ is a strictly CR-plurisubharmonic function on $(0,n-1-q)$-forms with CR-plurisubharmonicity constant $\Alm$
Then there exists a constant $C$ that is independent of $\Alm$ so that
\[
\Qbm(\tz\Psml\vp,\tz\Psml\vp) + C\|\tz\Psml\vp\|_\lp^2 + O_\lm(\|\tz\tPsol\vp\|_0^2) 
\geq \Alm \|\tz\Psml\vp\|_{\lm}^2.
\]
\end{prop}

\begin{proof} Similarly to the proof of Lemma \ref{prop:local results for Pspl}, we can apply Lemma \ref{lem:controlling Qbm}
to $\tz\Psml\vp$ which gives (for some $1 \gg \ep >0$)
\begin{align*}
&\Qbm(\tz\Psml\vp,\tz\Psml\vp) \geq (1-\ep') \sum_{J\in\I_q'}\sum_{j=1}^{n-1} \|\Lbam_j\tz\Psml\vp_J\|_\lm^2 \\
& + \Rre\Big\{\sum_{J\in\I_q'}\sum_{j\in J}  (c_{jj}(-T)\tz\Psml\vp_J,\tz\Psml\vp_{J})_{\lm} 
+ \sum_{J,J'\in \I_q'}\sum_{\atopp{1\leq j,k\leq n-1}{j\neq k}} \ep^{kJ}_{jJ'}
(c_{jk}(-T)\tz\Psml\vp_J,\tz\Psml\vp_{J'})_{\lm} \Big\} \\ 
\\ & + \frac12 \sum_{J\in\I_q'}\sum_{j\in J}\Big[ 
\big( (\Lb_j L_j(\lp)+L_j\Lb_j(\lp))\tz\Psml\vp_J,\tz\Psml\vp_{J})_\lp \big)\\
&\hspace{2in} + \sum_{\ell=1}^{n-1}\big( (d_{jj}^\ell L_\ell(\lm) 
+ \bar d_{jj}^\ell \bar L_\ell(\lm))\tz\Psml\vp_J,\tz\Psml\vp_{J}\big)_\lm\Big] \\
&+ \frac 12\sum_{J,J'\in \I_q'}\sum_{\atopp{1\leq j,k\leq n-1}{j\neq k}} \ep^{kJ}_{jJ'}
\bigg[\big( (\Lb_k L_j(\lm)+L_j\Lb_k(\lm))\tz\Psml\vp_J,\tz\Psml\vp_{J'})_\lm \big) \\
&\hspace{1in} + \sum_{\ell=1}^{n-1}\big( (d_{jk}^\ell L_\ell(\lm) 
+ \bar d_{kj}^\ell \bar L_\ell(\lm))\tz\Psml\vp_J,\tz\Psml\vp_{J'}\big)_\lm
 \bigg] + O(\|\tz\Psml\vp\|_0^2).
\end{align*}

To control the $T$ terms, we use Lemma \ref{lem:T bound lm}  since $\supp\tz\subset U'$, and
the Fourier transform of $\tz\Psml\vp$ is supported in $\Cm$ up to a smooth term. Indeed, with $A=A_0$ where $A_0$ is
from the definition of CR-plurisubharmonicity on $(0,q)$-forms,
\begin{multline*} 
\Rre\Big\{\sum_{J\in\I_q'}\sum_{j\in J}  (c_{jj}(-T)\tz\Psml\vp_J,\tz\Psml\vp_{J})_{\lm} 
+ \sum_{J,J'\in \I_q'}\sum_{\atopp{1\leq j,k\leq n-1}{j\neq k}} \ep^{kJ}_{jJ'}
(c_{jk}(-T)\tz\Psml\vp_J,\tz\Psml\vp_{J'})_{\lm} \Big\}\\
\geq A_0 \Big[\sum_{J\in\I_q'}\sum_{j\in J}  (c_{jj}\tz\Psml\vp_J,\tz\Psml\vp_{J})_{\lm} 
+ \sum_{J,J'\in \I_q'}\sum_{\atopp{1\leq j,k\leq n-1}{j\neq k}} \ep^{kJ}_{jJ'}
(c_{jk}\tz\Psml\vp_J,\tz\Pspl\vp_{J'})_{\lm} \Big] \\
+ O(\|\tz\Psml\vp\|_\lm^2) + O_\lm(\|\tz\tPsol\vp\|_0^2) 
\end{multline*}
Putting these estimates together, we have
\begin{multline*}
\Qbm(\tz\Psml\vp,\tz\Psml\vp) \geq \\ 
\sum_{J\in\I_q'}\sum_{j\in J}  (s_{jj}^-\tz\Psml\vp_J,\tz\Psml\vp_{J})_{\lm} 
+ \sum_{J,J'\in \I_q'}\sum_{\atopp{1\leq j,k\leq n-1}{j\neq k}} \ep^{kJ}_{jJ'}
(\sjkm\tz\Psml\vp_J,\tz\Psml\vp_{J'})_{\lm}  \\
+ O(\|\tz\Psml\vp\|_\lm^2) + O_\lm(\|\tz\tPsol\vp\|_0^2) ).
\end{multline*}
Recall that $\lm$ is strictly CR-plurisubharmonic on $(0,n-1-q)$-forms with 
CR-plurisubharmonicity constant $A_\lm$.
In local coordinates, if $L = \sum_{j=1}^{n-1} \xi_j L_j$, then
\[
\Big\la \frac12\big(\p_b\dbarb\lm - \dbarb\p_b\lm\big) + A_0 d\gamma, L\wedge\Lb\Big\ra
 = \sum_{j,k=1}^{n-1} \sjkm \xi_j\bar\xi_k,
\]
and $(\sjkm)$ is a Hermitian matrix. Therefore, by the multilinear algebra lemmas,
Lemma \ref{lem:multilinear algebra} and Lemma \ref{lem:multilin for n-1-q at a time},
\[
\Qbp(\tz\Pspl\vp,\tz\Pspl\vp) + C\|\tz\Pspl\vp\|_\lp^2 + O_\lp(\|\tz\tPsol\vp\|_0^2) 
\geq A_\lp \|\tz\Pspl\vp\|_{\lp}^2.
\]
where the constant $C$ is independent of $A_\lp$.
\end{proof}

We are finally ready to prove the basic estimate.
\begin{proof}[Proof. (Basic Estimate -- Proposition \ref{prop:basic estimate})]
From \eqref{eqn:energy form -- dbarb commuted by psi-do}, there exist constants $K$, $K_\pm$ so that if
$A_\pm = \min\{A_\lm,A_\lp\}$, then
\begin{multline*}
K\Qbpm(\vp,\vp) + K_\pm\sumn \|\tzn\tPsoln\zn\vpn\|_0^2 + K'\|\vp\|_0^2 + O_{\pm}(\|\vp\|_{-1}^2) \\
\geq \sumn \Big[ \Qbp(\tzn\Pspln\zn\vpn,\tzn\Pspln\zn\vpn) 
 + \Qbm(\tzn\Psmln\zn\vpn,\tzn\Psmln\zn\vpn) \Big].
\end{multline*}
From Proposition \ref{prop:local results for Pspl} and Proposition \ref{prop:local results for Psml} it follows that
by increasing the size of $K$, $K_\pm$, and $K'$ (where $K'$ does NOT depend on $A$) that
\[
K\Qbpm(\vp,\vp) + K_\pm\sumn \|\tzn\tPsoln\zn\vpn\|_0^2 +  K'\|\vp\|_0^2 + O_{\pm}(\|\vp\|_{-1}^2)
\geq A_\pm \|\vp\|_0^2
\]
\end{proof}

%
\subsection{A Sobolev estimate in the ``elliptic directions"}
For forms whose Fourier transforms are supported up to a smooth term in $\Co$, we have better estimates. 
The following result is the $(0,q)$-form version of 
Lemma 4.18 in \cite{Nic06}.


\begin{lem}\label{lem: Co supported terms are benign}
Let $\vp$ be a $(0,1)$-form supported in $U_\nu$ for some $\nu$ such that up to a smooth term,
$\hat\vp$ is supported in $\tCon$. There exist positive constants $C>1$ and $C_1>0$ independent of $A$ so that
\[
C \Qbpm(\vp,E_\pm\vp) + C_1\|\vp\|_0^2 \geq \|\vp\|_1^2.
\]
\end{lem}
The proof in \cite{Nic06} also holds  at level $(0,q)$.

We can use 
Lemma \ref{lem: Co supported terms are benign}
to control terms of the form $\|\tzn \Psoln \zn\vpn\|_0^2$.
\begin{prop}\label{prop:controlling the elliptic terms}
For any $\ep>0$, there exists $C_{\ep,\pm}>0$ so that 
\[
\|\tzn \Psoln \zn\vpn\|_0^2 \leq \ep \Qbpm(\vpn,\vpn) + C_{\pm} \|\vpn\|_{-1}^2.
\]
\end{prop}

\begin{proof} Observe that
$\|\tzn \Psoln \zn\vpn\|_0^2 = \|\Lambda^{-1} \tzn\Psoln \zn\vpn\|_1^2$. The $(0,q)$-form $\tzn\Psoln \zn\vpn$ is supported
in $\Co$, so Lemma \ref{lem: Co supported terms are benign} applies. Although the range of $\Lambda^{-1}$ is outside
$U_\nu$, we can write $\Lambda^{-1}\tzn = \zn' \Lambda^{-1} \tzn + (1-\zn')\Lambda^{-1}\tzn$ where $\zn'$ is a smooth bump function that is
identically one on the support of  $\tzn$. Then $(1-\zn')\Lambda^{-1}\tzn$ is infinitely smoothing and hence can be absorbed in the
$\|\vp\|_{-1}^2$ term. Let $P = \zn'\Lambda^{-1}$ and $\psi = \tzn\Psoln \zn\vpn$.
By Lemma \ref{lem: Co supported terms are benign} and the fact that $P$ is an order -1 pseudodifferential operator,
\[
\|\Lambda^{-1} \tzn\Psoln \zn\vpn\|_1^2 \leq \|P\psi\|_1^2 +  C\|\vpn\|_{-1}^2
\leq C_1\Qbpm(P\psi,P\psi) +   C\|\vpn\|_{-1}^2.
\]
The adjoint of $P$ is 
$P^{*,\pm} = \zn' \Lambda^{-1}$. Consequently $P-P^{*,\pm}$ is an order -2 pseudodifferential operator, and we can apply 
Lemma 2.4.2 in \cite{FoKo72} to prove
\[
\Qbpm(P\psi,P\psi) = \Rre\Qbpm(\psi,P^{*,\pm}P\psi) + C_{\pm} \|\vpn\|_{-1}^2
\leq \ep \Qbpm(\vpn,\vpn)  + C_{\ep,\pm} \|\vpn\|_{-1}^2.
\]
\end{proof}
The term $ \ep \Qbpm(\vp,\vp)$ could be replaced by $ \ep\|\Boxbpm \vp\|_{-1}^2$ if we had a need for it.

%
%
\section{Existence and Compactness Theorems for the Complex Green Operator}
\label{sec:complex Green operator}

In this section, we use the basic estimate to prove existence and compactness theorems for the complex Green operator.
As always, $M$ is a compact, orientable, weakly pseudoconvex CR-manifold of dimension at least 5, endowed with 
strongly CR-plurisubharmonic functions $\lp$ and $\lm$. 

%
%
\subsection{Closed range for $\Boxbpm$.}

For $1\leq q\leq n-2$, let 
\begin{align*}
\Hpm^q &= \{\vp\in\Dom(\dbarb)\cap\Dom(\dbarbs) : \dbarb\vp=0, \dbarbspm\vp=0\} \\ 
& = \{\vp\in\Dom(\dbar)\cap\Dom(\dbarbs):\Qbpmp=0\}
\end{align*}
be the space of $\pm$-harmonic $(0,q)$-forms.

\begin{lem}\label{lem:Qbpmp controls norm vp}
For $\Apm$ suitably large and $1\leq q \leq n-2$, 
$\Hpm^q$ is finite dimensional and there exists $C$ that does not depend on $\lp$ and $\lm$ so that
for all $(0,q)$-forms $\vp\in\Dom(\dbarb)\cap\Dom(\dbarbs)$ so that $\vp\perp\Hpm^q$ (with respect to  $\la\cdot,\cdot\rapm$).
\begin{equation}\label{eqn:Qbpmp controls norm}
\norm\vp\normpm^2 \leq C\Qbpmp.
\end{equation}
\end{lem}

\begin{proof}
For $\vp\in\Hpm$, we can use Proposition \ref{prop:basic estimate} with $\Apm$ suitably large (to absorb terms) so that
\[
\Apm\norm\vp\normpm^2 \leq C_{\pm}\big( \sumn \|\tzn\Psoln\zm\vpn\|_0^2 + \|\vp\|_{-1}^2\big).
\]
Also, by Proposition \ref{prop:controlling the elliptic terms},
\[
\sumn \|\tzn\Psoln\zm\vpn\|_0^2 \leq C_{\pm} \|\vp\|_{-1}^2.
\]
since $\Qbpm(\vp,\vp)=0$.
The unit ball in $\Hpm\cap L^2(M)$ is compact, and hence finite dimensional.

Assume that \eqref{eqn:Qbpmp controls norm} fails. Then there exists $\vp_k \perp\Hpm$ 
with $\norm\vp_k\normpm =1$ so that
\[
\norm\vp_k\normpm^2 \geq k\Qbpm(\vp_k,\vp_k).
\]
For $k$ suitably large, we can use Proposition \ref{prop:basic estimate} and the above argument to absorb $\Qbpm(\vp_k,\vp_k)$ by
$A_\pm \norm \vp_k\normpm$ to get:
\[
\norm\vp_k\normpm^2 \leq C_{\pm}\|\vp_k\|_{-1}^2.
\]
Since $H^{-1}(M)$ is compact in $L^2(K)$,  there exists a subsequence $\vp_{k_j}$ that converges in $L^2(M)$.
Since $(\Qbpm(\cdot,\cdot) + \norm\cdot\normpm^2)^{1/2}$ is a norm that dominates the $L^2(M)$-norm, there is a further
subsequence that converges in the $(\Qbpm(\cdot,\cdot) + \norm\cdot\normpm^2)^{1/2}$ norm as well. The limit $\vp$
satisfies $\norm\vp\normpm=1$ and $\vp\perp\Hpm$. But from the above inequality, $\vp\in\Hpm$. This is a contradiction
and \eqref{eqn:Qbpmp controls norm} holds.
\end{proof}

Let 
\[
\Hpmp^q = \{\vp\in L^2_{0,q}(M) : \la \vp,\phi\rapm =0,\text{ for all }\phi\in\Hpm^q\}.
\]
On $\Hpmp^q$, define
\[
\Boxbpm = \dbarb\dbarbspm + \dbarbspm\dbarb.
\]
Since $\dbarbspm = E_{\pm}\dbarbs + [\dbarbs,E_{\pm}]$, $\Dom(\dbarbspm) = \Dom(\dbarbs)$. This causes
\[
\Dom(\Boxbpm) = \{\vp\in L^2_{0,q}(M) : \vp\in\Dom(\dbarb)\cap\Dom(\dbarbs),\ \dbarb\vp\in\Dom(\dbarbs),\text{ and }
\dbarbs\vp\in\Dom(\dbarb)\}.
\]

%
%
\subsection{Proof of Theorem \ref{thm:G_q is compact} when $s=0$.}
This subsection is devoted the proof of Theorem \ref{thm:G_q is compact} when $s=0$, i.e., the $L^2$-case. 

As a consequence of Lemma \ref{lem:Qbpmp controls norm vp}, we may apply Theorem 1.1.2 in \cite{Hor65} to conclude that
$\dbarb:L^2_{(0,q)}(M) \to L^2_{(0,q+1)}(M)$ and  $\dbarbspm:L^2_{(0,q)}(M)\to L^2_{(0,q-1)}(M)$ have closed range. However, by
Theorem 1.1.1 in \cite{Hor65}, this also means that 
$\dbarb:L^2_{(0,q-1)}(M) \to L^2_{(0,q)}(M)$ and  $\dbarbspm:L^2_{(0,q+1)}(M)\to L^2_{(0,q)}(M)$ have closed range (and satisfy
the appropriate $L^2$ inequality with a constant that does NOT depend on $\lp$ or $\lm$).
Again by Lemma \ref{lem:Qbpmp controls norm vp}, Theorem 1.1.1 in \cite{Hor65}, and Lemma
\ref{lem:equivalence of norms}, $\dbarbs$ has closed range  when acting on $L^2_{(0,q)}(M)$ or $L^2_{(0,q+1)}(M)$. 
Therefore, for a $(0,q)$-form $u\in\Dom(\dbarb)\cap\Dom(\dbarbs)$, we have the estimates
\begin{equation}\label{eqn:l2 bound for q}
\|u\|_0^2 \leq C(\|\dbarb u\|_0^2 + \|\dbarbs u\|_0^2 + \| H_q u\|_0^2)
\end{equation}
and
\begin{equation}\label{eqn:l2 bound for qpm}
\|u\|_0^2 \leq C(\Qbpm(u,u) + \| H_{\pm,q} u\|_0^2)
\end{equation}
where $H_q$ is the projection of $u$ onto $\H^q$ and $H_{\pm,q}$ is the projection of $u$ onto $\Hpm^q$.
This implies the existence of $G_q$ and $\Gpmq$ as  bounded
operators on $L^2_{(0,q)}(M)$ that invert $\Boxb$ on $\H^q$ and $\Boxbpm$ on $\Hpm^q$, respectively
(see for example \cite{Shaw85a}, Lemma 3.2 and its proof).
Moreover, the solvability of $\dbarb$ in $L^2_{(0,q)}(M)$ and weighted $L^2_{(0,q)}(M)$ forces
\[
\ker(\dbarb) = \underbrace{\Ran(\dbarb) \oplus \Hpm^q}_{\oplus\text{ with respect to } \la\cdot,\cdot\rapm}
= \underbrace{\Ran(\dbarb) \oplus \H^q}_{\oplus\text{ with respect to }(\cdot,\cdot)_0}.
\]
Consequently, $\H^q$ is finite dimensional.

We now prove that $G_q$ is compact. 
First observe, we have the following identity:
\begin{align*}
G_{q+1} \dbarb u = G_{q+1} \dbarb(\dbarb\dbarbs + \dbarbs\dbarb)G_q u
&= G_{q+1} \dbarb\dbarbs\dbarb G_q u \\
&= G_{q+1}(\dbarb\dbarbs+\dbarbs\dbarb)\dbarb G_q u = \dbarb G_q u.
\end{align*}
Thus,
\[
\dbarb G_q = (\dbarbs G_{q+1})^*.
\]
To prove compactness of $G_q$, it suffices to show compactness on $\Hp^q$ (since $G_q$ is zero on $\H^q$).
When $u\in \Hp^q$, equation \eqref{eqn:l2 bound for q} implies (since $G_q u \in\Hp^q$)
\begin{equation}\label{eqn:Gq decomp}
\|G_q u\|_{0}^2 \les \|\dbarb G_q u\|_0^2 + \|\dbarbs G_q u\|_0^2
 = \|(\dbarbs G_{q+1})^* u\|_0^2 + \|\dbarbs G_q u\|_0^2.
\end{equation}
Therefore, we only need to show that both $\dbarbs G_q$ and
$\dbarbs G_{q+1}$ are compact. 
Our main tool will be a strengthening of \eqref{eqn:l2 bound for qpm}. We claim that
\begin{equation}\label{eqn:almost cptness l2 bound for qpm}
\| u\|_0^2 \leq \frac C\Apm\big( \norm\dbarb u\normpm^2 + \norm\dbarbspm u\normpm^2\big) + C_\pm \|u\|_{-1}^2.
\end{equation}
To prove \eqref{eqn:almost cptness l2 bound for qpm}, we already know the estimate if
$u\in \Hpm^q$, so we can assume that $u\in\Hpmp^q$.
we use Proposition \ref{prop:basic estimate} to see that 
\[
\Apm \norm u \normpm \leq K \Qbpm(u,u) + K_{\pm}\big( \sum_\nu \|\tzn\Psoln\zn u^\nu\|_0^2 + \| u \|_{-1}^2\big).
\]
Thus, to prove \eqref{eqn:almost cptness l2 bound for qpm}, we have to show that 
$K_{\pm} \sum_\nu \|\tzn\Psoln\zn u^\nu\|_0^2$ is well-controlled. 
Using Proposition \ref{prop:controlling the elliptic terms}, we have (with $\ep  = 1/K_\pm$), 
\[
K_{\pm} \sum_\nu \|\tzn\Psoln\zn u^\nu\|_0^2
\leq  \Qbpm(u,u) + K_{\pm}' \| u\|_{-1}^2.
\]
and \eqref{eqn:almost cptness l2 bound for qpm} is proved.

When $\alpha\in \Ran(\dbarb)\subset L^2_{0,q+1}(M)$, 
$\dbarbs G_{q+1}\alpha$ gives the norm minimizing solution to $\dbarb v=\alpha$,
$\alpha\in\Ran(\dbarb)\subset L^2_{0,q+1}(M)$, while $\dbarbspm G_{\pm,q+1}\alpha$ gives a different solution
(the one that minimizes the $\norm\cdot\normpm$-norm). For such $\alpha$, 
\eqref{eqn:almost cptness l2 bound for qpm} therefore implies
\begin{multline} \label{eqn:dbarbs Gq+1 bound}
\|\dbarbs G_{q+1}\alpha\|_0^2 \leq \|\dbarbspm G_{\pm, q+1}\alpha\|_0^2
\leq C \norm\dbarbspm G_{\pm, q+1}\alpha\normpm^2 \\
\leq \frac C{A_\pm}\norm \alpha \normpm + C_{\pm} \|\dbarbspm G_{\pm,q+1} \alpha\|_{-1}^2 
\leq \frac C{A_\pm}\norm \alpha \norm_\pm + C_{\pm} \|\dbarbspm G_{\pm,q+1} \alpha\|_{-1}^2
\end{multline}
Applying Lemma \ref{lem:Qbpmp controls norm vp} to $\dbarbspm G_{\pm,q+1}$ shows that 
$\dbarbspm G_{\pm,q+1}: L^2_{0,q+1}(M) \to L^2_{0,q+1}(M)$ is a bounded operator with $C$ is independent of $A_{\pm}$. Therefore, 
$L^2_{0,q+1}(M)$ embeds compactly in  $W^{-1}_{0,q+1}(M)$. 
Moreover, $A_\pm$ can be made arbitrarily large since $M$ satisfies $(P_q)$ and
$(P_{n-1-q})$.
Equation  \eqref{eqn:dbarbs Gq+1 bound}  proves that
$\dbarbspm G_{\pm,q+1}: L^2_{0,q+1}(M) \to L^2_{0,q}(M)$ continuously, so the map 
$\dbarbspm G_{\pm,q+1}: L^2_{0,q+1}(M) \to
W^{-1}_{0,q}(M)$ is compact, and it follows that $\dbarbs G_{q+1}$ is compact on $\Ran(\dbarb)$ by
\cite{DAngelo02}, Proposition V.2.3. On the orthogonal complement of $\Ran(\dbarb)$, $\dbarbs G_{q+1}=0$, so
$\dbarbs G_{q+1}:L^2_{0,q+1}(M)\to L^2_{0,q+1}(M)$ is compact. To estimate $\dbarbs G_q\alpha$, we cannot
invoke \eqref{eqn:almost cptness l2 bound for qpm} directly because $\dbarbs G_q\alpha$ is a $(q-1)$-form.
Instead, for $\alpha\in\Ran(\dbarb)\subset L^2_{0,q}(M)$,
\begin{align}
&\norm \dbarbspm G_{\pm,q}\alpha\normpm^2 = \la \dbarb\dbarbspm G_{\pm,q}\alpha, G_{\pm,q}\alpha\rapm
= \la \alpha, G_{\pm,q}\alpha \rapm \nn \\
&\leq \frac {2C}{A_{\pm}}\norm\alpha\normpm^2 + \frac{A_{\pm}}{2C} \norm G_{\pm,q}\alpha\normpm^2
\leq \frac{2C}{A_{\pm}}\norm\alpha\normpm^2 + \frac 12 \norm \dbarbspm G_{\pm,q}\alpha\normpm^2 + C_\pm \|G_{\pm,q}\alpha\|_{-1}^2.
\label{eqn:dbarbs Gq bound}
\end{align}
Here we have used that $\dbarb\alpha=0$ and that $\alpha \in \Hpmp^q$ (since $\alpha\in\Ran\dbarb$) in the second inequality.
Also, the first inequality shows that the $\norm \dbarbspm G_{\pm,q}\alpha\normpm^2<\infty$ and thus the term in the final
inequality can be absorbed. Thus we can can prove $\dbarbs G_q L^2_{0,q}(M)\to  L^2_{0,q-1}(M)$ 
is a compact operator by repeating the argument
that follows \eqref{eqn:dbarbs Gq+1 bound} with $G_{\pm,q}$ replacing $\dbarbspm G_{\pm,q+1}$. 

%
%
\subsection{End proof of Theorem \ref{thm:G_q is compact} -- the $s>0$ case.}
Fix $s>0$. Recall that 
compactness  $G_q$ in $L^2_{0,q}(M)$ is equivalent to the
following compactness estimate: for every $\ep>0$, there exists $C_\ep>0$  so that for every $u\in \Dom(\dbarb)\cap\Dom(\dbarbs)$,
\[
\| u \|_{0}^2 \leq \ep(\|\dbarb u\|_0^2 + \|\dbarbs u\|_0^2) + C_\ep \|u\|_{-1}^2.
\]
We claim that this estimate also holds a priori in $H^s$, $s>0$. Indeed,
using the fact that the commutators $[\dbarb,\Lambda^s]$ and $[\dbarbs,\Lambda^s]$ are pseudodifferential operators of order
$s$ (independent of $\ep$), we have
\begin{align*}
\| u \|_s^2 &= \|\Lambda^s u\|_0^2 \leq \ep (\|\dbarb\Lambda^s u\|_0^2 + \|\dbarbs\Lambda^s u\|_0^2) + C_\ep \|\Lambda^s u\|_{-1}^2 \\
&\leq \ep(\|\Lambda^s\dbarb u\|_0^2 + \|\Lambda^s\dbarbs u\|_0^2) + \ep(\|[\dbarb,\Lambda^s] u\|_0^2 + \|[\dbarbs,\Lambda^s] u\|_0^2)
C_\ep \|u\|_{s-1}^2 \\
&\leq \ep(\|\dbarb u\|_s^2 + \|\dbarbs u\|_s^2) + C\ep\|u\|_s^2 + C_\ep\|u\|_{s-1}^2.
\end{align*}
When $\ep<1/2C$, the $C\ep\|u\|_s^2$ can be absorbed into the left-hand side of the equation. 
Thus, we have the estimate that for every $\ep>0$, there exists $C_\ep>0$  so that for every 
$u\in H^s_{0,q}(M)$ with $\dbarb u \in H^s_{0,q+1}(M)$ and $\dbarbs u \in H^s_{0,q-1}(M)$,
\begin{equation}\label{eqn:unregularized cptness estimate in Hs}
\| u \|_{s}^2 \leq \ep(\|\dbarb u\|_s^2 + \|\dbarbs u\|_s^2) + C_\ep \|u\|_{s-1}^2.
\end{equation}

Unlike in $L^2$-case, this estimate does not imply that $G_q$ is compact in $H^s$. The difficulty rests in the fact that while
$u$ may be in $H^s_{0,q}(M)$, we can only say that $G_q u \in L^2_{0,q}(M)$. We need to work with the family of 
regularized operators $G_{\delta, q}$, $0<\delta\leq 1$, arising from the following regularization. Let $\Qbo^\delta(\cdot,\cdot)$
be the quadratic form on $H^1_{0,q}(M)$ defined by
\[
\Qbo^\delta(u,v) = \Qbo(u,v) + \delta Q_L(u,v)
\]
where $Q_L$ is the hermitian inner product associated to the de Rham exterior derivative $d$, i.e.,
$Q_L(u,v) = (du,dv)_0 + (d^*u,d^*v)_0$. The inner product $Q_L$ has form domain
$H^1_{0,q}(M)$. Consequently, $\Qbo^\delta$ gives rise a unique, self-adjoint, elliptic operator $\Box_{b,\delta}$ with
inverse $G_{q,\delta}$. Equivalently, for $u\in L^2_{0,q}(M)$ and $v\in H^1_{0,q}(M)$, $(u,v)_0 = \Qbo^\delta(G_{q,\delta}u,v)$.
By elliptic regularity, we know that if $u\in H^s_{0,q}(M)$, then $G_{q,\delta} u\in H^{s+2}_{0,q}(M)$. We claim that for any 
$\ep>0$, there exists $C_\ep$ so that for any $u\in H^s_{0,q}(M)$,
\begin{equation}\label{eqn:cptness of Gqdelta}
\|G_{q,\delta} u \|_{s}^2 \leq \ep \|u\|_s^2 + C_\ep \|u\|_{s-1}^2,
\end{equation}
where the inequalities are uniform in $0<\delta\leq 1$.
Estimates of the form \eqref{eqn:cptness of Gqdelta} are well known to be equivalent to the compactness of $G_{q,\delta}$ on $H^s_{0,q}(M)$,
(see, for example, \cite{DAngelo02}, Proposition V.2.3).

By the a priori estimate \eqref{eqn:unregularized cptness estimate in Hs},
\[
\|G_{q,\delta} u \|_{s}^2 \leq \ep(\|\dbarb G_{q,\delta}u\|_s^2 + \|\dbarbs G_{q,\delta}u\|_s^2) + C_\ep \|u\|_{s-1}^2.
\]
The $\dbarb$ and $\dbarbs$ terms can be estimated as follows:
\begin{align*}
\|\dbarb G_{q,\delta}u\|_s^2 + \|\dbarbs G_{q,\delta}u\|_s^2
&\leq \Qbo(\Lambda^s G_{q,\delta} u, \Lambda^s G_{q,\delta} u) + C \|G_{q,\delta} u\|_s^2 \\
&\leq \Qbo^\delta(\Lambda^s G_{q,\delta} u, \Lambda^s G_{q,\delta} u) + C \|G_{q,\delta} u\|_s^2 \\
&\leq |(\Lambda^s u, \Lambda^s G_{q,\delta} u)_0| + C \| u\|_s^2,
\end{align*}
where we have used the estimate $\Qbo^\delta(\Lambda^s G_{q,\delta} u, \Lambda^s G_{q,\delta} u)
\leq |(\Lambda^s u, \Lambda^s G_{q,\delta} u)_0|  + C \|G_{q,\delta} u\|_s^2$, which follows from \cite{KoNi65}, Lemma 3.1. Thus,
we have
\[
\| G_{q,\delta} u \|_s^2 \leq \ep (\|G_{q,\delta} u\|_s^2 + \|u\|_s^2) + C_\ep \|u\|_{s-1}^2,
\]
By absorbing terms (and choosing $\ep<1/2$), we have proven (\ref{eqn:cptness of Gqdelta}) with the constant $C_\ep$ 
independent of $\delta$, $0<\delta\leq 1$.

We want to let $\delta\to 0$. If $u\in H^s_{0,q}(M)$, then $\{G_{q,\delta} u: 0<\delta\leq 1\}$ is bounded
in $H^s_{0,q}(M)$. Thus, there exists a sequence $\delta_k\to 0$ and $\tilde u\in H^s_{0,q}(M)$ so that
$G_{q,\delta_n} u \to \tilde u$ weakly in $H^s_{0,q}(M)$. Consequently, if $v\in H^1_{0,q}(M)$, then
\[
\lim_{n\to\infty} \Qbo^{\delta_n}(G_{q,\delta_n}u,v) = \Qbo(\tilde u,v).
\]
However, 
\[
\Qbo^{\delta_n}(G_{q,\delta_n} u, v) = (u,v) = \Qbo(G_q u,v),
\]
so $G_q u = \tilde u$ and \ref{eqn:cptness of Gqdelta} is satisfied with $\delta=0$. Thus, $G_q$ is a compact operator on $H^s_{0,q}(\Omega)$.\
and Theorem \ref{thm:G_q is compact} is proved.


%
%
%
%
%
%
%
%
\appendix
%
%

\section{Multilinear Algebra} \label{app:multilinear algebra}

Some crucial multilinear algebra is contained in the following lemma from Straube \cite{Str09}.
\begin{lem}\label{lem:multilinear algebra}
Let $(\lam_{jk})_{j,k=1}^m(z)$ be an $m\times m$ matrix-valued function and $1\leq q \leq m$. The following are equivalent:
\begin{enumerate}
\item $\displaystyle\sum_{K\in\I_{q-1}} \sum_{j,k=1}^m \lam_{jk}(z) u_{jK} \overline{u_{kK}} \geq A|u|^2\,\quad
\forall u\in\Lambda_z^{(0,q)}$.

\item The sum of any $q$ eigenvalues of $(\lam_{jk}(z))_{j,k}$ is at least $A$.

\item For any orthonormal $\underline{t}^\ell\in\C^m$, $1\leq j\leq q$,
\[
\sum_{\ell=1}^q \lam_{jk}(z) (\underline{t}^\ell)_j \overline{(\underline{t}^\ell)_k} \geq A
\]
\end{enumerate}
\end{lem}

These are Lemma 6.3 and Lemma 6.4 in \cite{Nic06}.

\begin{lem}\label{lem:multilin for q at a time}
Let $(b_{jk})$ be a Hermitian matrix and let $1\leq q\leq n-2$. Then then $\binom{n-1}q$ by $\binom{n-1}q$ matrix 
$(B^q_{JJ'})$ given by
\begin{align*}
B^q_{JJ} &= \sum_{j\in J} b_{jj} \\
B^q_{JJ'} &= -\sum_{\atopp{1\leq j,k\leq n-1}{j\neq k}} \ep^{kJ}_{jJ'} b_{jk} \qquad \text{if }J\neq J',
\end{align*}
where $J$ and $J'$ are multiindices, $|J|=|J'|=q$ is also Hermitian. Moreover, the eigenvalues
of $(B^q_{JJ'})$ are sums of the eigenvalues of $(b_{jk})$ taken $q$ at a time.
\end{lem}

\begin{lem}\label{lem:multilin for n-1-q at a time}
Let $(d_{jk})$ be a Hermitian matrix and let $1\leq q\leq n-2$. Then then $\binom{n-1}q$ by $\binom{n-1}q$ matrix 
$(D^q_{JJ'})$ given by
\begin{align*}
D^q_{JJ} &= \sum_{j\in J} b_{jj} \\
D^q_{JJ'} &= \sum_{\atopp{1\leq j,k\leq n-1}{j\neq k}} \ep^{kJ}_{jJ'} b_{jk} \qquad \text{if }J\neq J',
\end{align*}
where $J$ and $J'$ are multiindices, $|J|=|J'|=q$ is also Hermitian. Moreover, the eigenvalues
of $(D^q_{JJ'})$ are sums of the eigenvalues of $(d_{jk})$ taken $n-1-q$ at a time, so
$(D^q_{JJ'})$ is positive definite if $(d_{jk})$ is positive definite and $n-1-q>0$;
$(D^q_{JJ'})$ is positive semi-definite if $(d_{jk})$ is positive semi-definite for any $n$.
\end{lem}

If $q=1$, then Lemma \ref{lem:multilin for n-1-q at a time} says that if
$n\geq 3$ and $H=(h_{jk})$ is a Hermitian, positive definite matrix,
$1\leq i,k\leq n-1$, then $(\delta_{jk}\sum_{\ell=1}^{n-1}h_{\ell\ell}-h_{jk})$ is a Hermitian, positive definite matrix.
The requirement that $n\geq 3$ is the seemingly technical reason that Theorem \ref{thm:G_q is compact} is stated for $2n-1\geq 5$, as
well as the results in \cite{Nic06} and the fact that the work by Kohn and Nicoara in \cite{KoNi06} assumes closed range of $\dbarb$.

\bibliographystyle{alpha}
\bibliography{mybib}

\end{document}